\documentclass[11pt]{amsart}
\usepackage[centertags]{amsmath}
\usepackage{amsfonts}
\usepackage{amssymb}
\usepackage{amsthm}
\usepackage{mathtools}
\usepackage{newlfont}
\usepackage{amscd}
\usepackage{mathrsfs}
\usepackage{lmodern}
\usepackage[french,english]{babel}
\usepackage{microtype}
\usepackage[all,cmtip]{xy}
\usepackage{tikz}
\usetikzlibrary{trees}
\usepackage[pagebackref=off,colorlinks]{hyperref}
\hypersetup{pdffitwindow=true,linkcolor=blue,citecolor=blue,urlcolor=cyan}
\usepackage{cite}
\usepackage{fancyhdr}
\usepackage{stmaryrd}

\usepackage{tikz-cd}
\tikzcdset{arrow style=Latin Modern, row sep/normal=1.35em, column sep/normal=1.8em, cramped}

\newcommand\cyr{%
\renewcommand\rmdefault{wncyr}%
\renewcommand\sfdefault{wncyss}%
\renewcommand\encodingdefault{OT2}%
\normalfont
\selectfont}
\DeclareTextFontCommand{\textcyr}{\cyr}

\newcommand{\mint}{{\times}\kern-0.89em{\int}}

\usepackage[toc, page] {appendix}

 \numberwithin{equation}{section}

\newtheorem{thm}{Theorem}[section]
\newtheorem{cor}[thm]{Corollary}
\newtheorem{lem}[thm]{Lemma}
\newtheorem{prop}[thm]{Proposition}

\newtheorem{assu}[thm]{Assumption}
\newtheorem{conj}[thm]{Conjecture}

\theoremstyle{definition}
\newtheorem{defn}[thm]{Definition}
\newtheorem{rem}[thm]{Remark}

\newtheorem{notation}[thm]{Notation}

\newcommand{\Sh}{\mathrm{Sh}}
\newcommand{\et}{\textup{ét}}
\newcommand{\N}{\mathbb{N}}
\newcommand{\Zp}{\mathbb{Z}_p}
\newcommand{\Fc}{\mathcal{F}}

\newcommand{\Oc}{\mathcal{O}}
\newcommand{\Ab}{\mathrm{Ab}}
\newcommand{\HH}{\mathrm{H}}
\newcommand{\Spec}{\mathrm{Spec}}

\begin{document}
\title{On the Iwasawa invariants of Kato's zeta elements for modular forms}
\author{Chan-Ho Kim}
\address{
(Chan-Ho Kim) Department of Mathematics and Institute of Pure and Applied Mathematics,
Jeonbuk National University,
567 Baekje-daero, Deokjin-gu, Jeonju, Jeollabuk-do 54896, Republic of Korea
}
\email{chanho.math@gmail.com}
\author{Jaehoon Lee}
\address{(Jaehoon Lee) KAIST, 291 Daehak-ro, Yuseong-gu, Daejeon 34141, Republic of Korea}
\email{jaehoon.lee900907@gmail.com}
\author{Gautier Ponsinet}
\address{(Gautier Ponsinet) Universit\"{a}t Duisburg–Essen, Fakult\"{a}t f\"{u}r Mathematik, Thea-Leymann-Str. 9, 45127 Essen, Germany}
\email{gautier.ponsinet@uni-due.de}
\thanks{Chan-Ho Kim was partially supported
by a KIAS Individual Grant (SP054103) via the Center for Mathematical Challenges at Korea Institute for Advanced Study,
by the National Research Foundation of Korea(NRF) grant funded by the Korea government(MSIT) (No. 2018R1C1B6007009, 2019R1A6A1A11051177, RS-2024-00339824),
by research funds for newly appointed professors of Jeonbuk National University in 2024, and
by Global-Learning \& Academic research institution for Master’s$\cdot$Ph.D. Students, and Postdocs (LAMP) Program of the National Research Foundation of Korea (NRF) funded by the Ministry of Education (No. RS-2024-00443714).
Jaehoon Lee is partially supported by KAIST Advanced Institute for Science-X Post-Doc Fellowship.
Gautier Ponsinet thanks the Max Planck Institut for Mathematics for its hospitality and support.
}

\date{\today}
\subjclass[2010]{11R23 (Primary); 11F33 (Secondary)}
\keywords{Iwasawa theory, Kato's Euler systems, Iwasawa main conjecture, congruences of modular forms}
\maketitle
\begin{abstract}
  We study the behavior under congruences of the Iwasawa invariants of the Iwasawa modules which appear in Kato's main conjecture without $p$-adic $L$-functions.
  It generalizes the work of Greenberg--Vatsal, Emerton--Pollack--Weston, B.D. Kim, Greenberg--Iovita--Pollack, and one of us simultaneously.
  As a consequence, we establish the propagation of Kato's main conjecture for modular forms of higher weight at \emph{arbitrary} good prime under the assumption on the mod $p$ non-vanishing of Kato's zeta elements.
  The application to the $\pm$ and $\sharp/\flat$-Iwasawa theory for modular forms is also discussed.
\end{abstract}
\setcounter{tocdepth}{1}
\tableofcontents

\section{Introduction}
\subsection{Overview}
Fix once and for all a prime number \(p \geq 5\).
In Iwasawa theory for elliptic curves and modular forms, the techniques of congruences of modular forms have played important roles.
Especially, in their ground-breaking work \cite{gv}, Greenberg and Vatsal observed that both algebraic and analytic Iwasawa invariants of elliptic curves with good ordinary reduction over the cyclotomic $\mathbb{Z}_p$-extension $\mathbb{Q}_{\infty}$ of $\mathbb{Q}$ can be described in terms of the information of the residual representations and the local behavior at bad reduction primes under the $\mu=0$ assumption.
As a consequence, it is proved that the Iwasawa main conjecture for an elliptic curve implies the Iwasawa main conjecture for all congruent elliptic curves when one-sided divisibility is given and $\mu=0$.
This fundamental idea has been generalized to \emph{a variety of} settings including Hida families, elliptic curves with supersingular reduction with $a_p=0$, arbitrary $\mathbb{Z}_p$-extensions including anticyclotomic $\mathbb{Z}_p$-extensions of imaginary quadratic fields with and without Heegner hypothesis, modular forms at non-ordinary primes, and their mixtures \cite{weston-iwasawa-invariants-galois-deformations, epw, ochiai-two-variable, greenberg-iovita-pollack, bdkim-supersingular-invariants,hachimori,pw-mu,choi-kim,kim-summary,castella-kim-longo,delbourgo-lei,kidwell-selmer,ponsinet-signed,hatley-lei-congruences,fouquet-wan}.

In this article, we give a rather different and unified approach to realize the idea of Greenberg--Vatsal for modular forms of higher weight at \emph{arbitrary} good primes over the \emph{full} cyclotomic extension $\mathbb{Q}(\zeta_{p^\infty})$ of $\mathbb{Q}$.
In order to do this, we directly work with Kato's zeta elements and Kato's main conjecture without $p$-adic $L$-functions \cite{kato-euler-systems}.
As corollaries, we can prove Greenberg--Vatsal type results for
both good ordinary and non-ordinary cases \emph{simultaneously} including
$\pm$-Iwasawa theory \`{a} la Kobayashi--Pollack \cite{kobayashi-thesis, pollack-thesis} and Lei \cite{lei-compositio}
 and $\sharp/\flat$-Iwasawa theory \`{a} la Sprung \cite{sprung-ap-nonzero} and Lei--Loeffler--Zerbes \cite{lei-loeffler-zerbes_wach}.
Note that the construction of \emph{integral} $p$-adic $L$-functions of modular forms depends genuinely on the reduction type and we do not have to consider this issue at all.

The key ingredients include an extensive use of the localization exact sequence in \'{e}tale cohomology as well as the mod $p$ multiplicity one and Ihara's lemma.
Since Fontaine--Laffaille theory is implicitly used to obtain the mod $p$ multiplicity one and Ihara's lemma, the weight $k$ of modular forms we consider is required to satisfy $2 \leq k \leq p-1$. The weight assumption can be removed in the semi-stable ordinary case thanks to the work of Vatsal.

We also point out where the classical argument breaks down and how to overcome this obstruction.
In the argument of Greenberg--Vatsal and its successors, the relevant Selmer group over the Iwasawa algebra has no proper Iwasawa submodule of finite index and it is essentially used to reveal the following phenomenon.
\begin{quote}
Let $f$ be a newform, $\overline{\rho}_f$ the residual representation, and $\Sigma_0$ a finite set of primes consisting of all bad reduction primes. In both algebraic and analytic sides, we have the following equality:
\begin{align*}
``\textrm{ $\lambda$-invariant of $f$ } = & \textrm{ $\lambda$-invariant of $\overline{\rho}_f$} \\
& + \textrm{ $\sum_{\ell \in \Sigma_0}$ $\lambda$-invariant of the local behavior of $f$ at $\ell$}"
\end{align*}
when $\mu=0$.
Here, the algebraic side means the Selmer group part and the analytic side means the $p$-adic $L$-function part in the Iwasawa main conjecture of Mazur--Greenberg type.
\end{quote}
In our setting, we do \emph{not} expect that the relevant Iwasawa modules have no finite Iwasawa submodule in general. Thus, the above formula would not hold as it stands in our setting. However, we are still able to prove the following type of statement:
\begin{quote}
Fix a 2-dimensional odd irreducible mod $p$ representation $\overline{\rho}$.
Let $S_k(\overline{\rho})$ be the set of newforms of weight $k$ such that the residual representation is isomorphic to $\overline{\rho}$ and $p$ does not divide the levels of the newforms. Suppose that the $\mu$-invariant of one form in  $S_k(\overline{\rho})$  in the zeta element side vanishes.
Then, for all $f \in S_k(\overline{\rho})$, we have the following statements:
\begin{itemize}
\item The $\mu$-invariants of $f$ in both the zeta element and $\mathrm{H}^2$-sides vanish;
\item $\textrm{($\lambda$-invariant of $f$ in the $\mathrm{H}^2$-side)} - \textrm{($\lambda$-invariant of $f$ in the zeta element side)}$ is constant.
\end{itemize}
Here, the zeta element side means the Kato's Euler system part and the $\mathrm{H}^2$-side means the the second Iwasawa cohomology part in the Iwasawa main conjecture of Kato type.
\end{quote}
In other words, although we are not able to see how Iwasawa invariants vary in each side under congruences, the stability of the difference of Iwasawa invariants is strong enough to deduce a Greenberg--Vatsal type result for Kato's main conjecture.

The following non-exhaustive list shows how much the idea of Greenberg--Vatsal is generalized. (See \S\ref{subsec:iwasawa-cohomologies} for the difference between $\mathbb{Q}(\zeta_{p^\infty})$ and $\mathbb{Q}_\infty$.)
\begin{itemize}
\item In \cite{gv}, elliptic curves with good ordinary reduction over $\mathbb{Q}_{\infty}$ are studied;
\item In \cite{epw}, Hida families over the full cyclotomic extension $\mathbb{Q}(\zeta_{p^\infty})$ are studied;
\item In \cite{bdkim-supersingular-invariants}, elliptic curves with good supersingular reduction ($a_p=0$) over $\mathbb{Q}_{\infty}$ are studied (the algebraic side only);
\item In \cite{greenberg-iovita-pollack}, modular forms of weight two at non-ordinary primes over $\mathbb{Q}_{\infty}$ are studied;
\item In \cite{ponsinet-signed}, modular forms of higher weight in the Fontaine--Laffaille range at non-ordinary primes  over $\mathbb{Q}_{\infty}$ are studied (the algebraic side only).
\end{itemize}
Indeed, some of the above work deal with more general $\mathbb{Z}_p$-extensions.
Our main theorem (Theorem \ref{thm:main-theorem}) covers and strengthens all the above results simultaneously for modular forms of weight $k$ with $2 \leq k \leq p-1$ over the full cyclotomic extension $\mathbb{Q}(\zeta_{p^\infty})$.

The rest of this article is organized as follows.
\begin{enumerate}
\item
In the rest of this section, we recall the convention of modular Galois representations, Iwasawa cohomology, Kato's zeta elements and the Iwasawa main conjecture without $p$-adic $L$-functions closely following \cite{kato-euler-systems}.
\item
In $\S$\ref{sec:main-results}, we state our main result (Theorem \ref{thm:main-theorem} and Theorem \ref{thm:main-theorem-ordinary}) and discuss the applications to other types of Iwasawa theory.
\item In $\S$\ref{sec:local}, we recall the notion of Iwasawa-theoretic Euler factors  at $\ell \neq p$ in various settings and study their Iwasawa invariants.
\item In $\S$\ref{sec:zeta-elements}, we study the mod $p$ behavior of the zeta elements. Theorem \ref{thm:main-theorem}.(1) on the $\mu$-invariant is proved here.
\item In $\S$\ref{sec:H^2}, we study the mod $p$ behavior of the second Iwasawa cohomology.
\item In $\S$\ref{sec:connection}, we prove Theorem \ref{thm:main-theorem}.(2) on the $\lambda$-invariant.
\end{enumerate}
\begin{notation}
We expect the reader is rather familiar with \cite{kato-euler-systems}. We freely use the results and notations of \cite{kato-euler-systems}.
Especially, we follow the notation of  \cite{kato-euler-systems} with only few exceptions.
We denote by $\pi$ instead of $\lambda$ for the uniformizer of coefficient fields in order to avoid the conflict with $\lambda$-invariants. We also denote by $V_f$ and $T_f$ instead of $V_{F_{\lambda}}(f)$ and $V_{\mathcal{O}_{\lambda}}(f)$ for modular Galois representations and their lattices, respectively.
\end{notation}
\subsection{Galois representations}
Let $f = \sum_{n \geq 1} a_n(f) q^n \in S_k(\Gamma_1(N), \psi)$ be a newform with $(N,p)=1$.
We fix embeddings
$\iota_p : \overline{\mathbb{Q}} \hookrightarrow \overline{\mathbb{Q}}_p$ and
$\iota_\infty : \overline{\mathbb{Q}} \hookrightarrow \mathbb{C}$.
Let
$F \coloneqq \mathbb{Q}(a_n(f):n)$ and $F_\pi \coloneqq \mathbb{Q}_p (\iota_p(a_n(f)):n)$, and
$\mathcal{O}_\pi \subseteq F_\pi$ the ring of integers of $F_\pi$, $\pi$ a uniformizer of $F_\pi$, and $\mathbb{F}$ the residue field of $F_\pi$.

Let $S$ be a finite set of places of $\mathbb{Q}$ containing the places dividing $Np$ and $\mathbb{Q}_S$ be the maximal extension of $\mathbb{Q}$ unramified outside $S$. For a field $K$, denote by $G_K$ the absolute Galois group of $K$.
Let
$$\rho_f : \mathrm{Gal}(\mathbb{Q}_S/\mathbb{Q}) \to \mathrm{GL}_2(F_{\pi}) \simeq \mathrm{GL}_{F_\pi}(V_f)$$
be the (cohomological) $\pi$-adic Galois representation associated to $f$ following the convention of \cite[$\S$14.10]{kato-euler-systems}.
\subsubsection{Construction} \label{subsubsec:construction}
We follow \cite[(4.5.1) and $\S$8.3]{kato-euler-systems}.
Let $N \geq 4$ and $\varpi : E \to Y_1(N)$ the universal elliptic curve over the modular curve and
$\mathcal{H}^1_p \coloneqq \mathrm{R}^1\varpi_*\mathbb{Z}_p$ the \'{e}tale $\mathbb{Z}_p$-sheaf on $Y_1(N)$.
We define
$$ V_{k,\mathbb{Z}_p}(Y_1(N)) \coloneqq \mathrm{H}^1_{\mathrm{\acute{e}t}}(Y_1(N)_{\overline{\mathbb{Q}}}, \mathrm{Sym}^{k-2}_{\mathbb{Z}_p}(\mathcal{H}^1_p)) ,$$
and
\[
\xymatrix{
T_f \coloneqq V_{k, \mathbb{Z}_p} (Y_1(N)) \otimes_{\mathbb{T}(N)} \mathbb{T}(N)/\wp_f , & V_f \coloneqq T_f \otimes_{\mathbb{Z}_p}\mathbb{Q}_p
}
\]
where
 $\mathbb{T}(N)$ is the image of the abstract Hecke algebra generated by Hecke operators at all primes in the endomorphism ring of $V_{k,\mathbb{Z}_p}(Y_1(N))$ over $\mathbb{Z}_p$ and $\wp_f$ is the height one prime ideal of $\mathbb{T}(N)$ generated by the Hecke eigensystem of $f$ following \cite[$\S$6.3]{kato-euler-systems}.
When $N <4$, we are still able to construct Galois representations using the $N \geq 4$ case and trace maps.
\subsubsection{Properties}
More explicitly, $\rho_f$ satisfies the following properties:
\begin{enumerate}
\item $\mathrm{det} ( \rho_f ) = \chi^{1-k}_{\mathrm{cyc}} \cdot \psi^{-1}$
where $\chi_{\mathrm{cyc}}$ is the cyclotomic character;
\item for any prime $\ell$ not dividing $Np$, we have
\begin{equation} \label{eqn:char-poly}
\mathrm{det} \left( 1- \rho_f ( \mathrm{Frob}^{-1}_\ell ) \cdot u  : \mathrm{H}^0(I_\ell,  V_f ) \right) = 1 - a_{\ell}(f) \cdot u +  \psi  (\ell) \cdot \ell^{k-1} \cdot u^2
\end{equation}
where $\mathrm{Frob}_\ell$ is the arithmetic Frobenius at $\ell$ in $G_{\mathbb{Q}_\ell} / I_\ell$ and $I_\ell$ is the inertia subgroup of $G_{\mathbb{Q}_\ell}$;
\item for the prime number $p$ lying under $\pi$,  we have
$$\mathrm{det} \left( 1- \varphi \cdot u  : \mathbf{D}_{\mathrm{cris}}  ( V_f ) \right) = 1 - a_{p}(f) \cdot u +  \psi (p) \cdot p^{k-1} \cdot u^2$$
where $\varphi$ is the Frobenius operator acting on $\mathbf{D}_{\mathrm{cris}}  ( V_f )$, Fontaine's crystalline Dieudonn\'{e} module associated to the restriction of $V_f$ to $G_{\mathbb{Q}_p}$ \cite{fontaine-semi-stable}.
\end{enumerate}
For any Galois module $M$ over $\mathbb{Z}_p$ and any integer $k \in \mathbb{Z}$, let $M(k) \coloneqq M \otimes_{\mathbb{Z}_p} \mathbb{Z}_p(k)$ be the $k$-th Tate twist of $M$.

Let $f^* = \sum_{n \geq 1} \overline{a_n(f)} q^n \in S_k(\Gamma_1(N), \overline{\psi})$ be the dual modular form of $f$ where $\overline{(-)}$ means the complex conjugation.
Due to the duality of modular Galois representations \cite[(14.10.1)]{kato-euler-systems}, we also have
\[
\xymatrix{
V_{f^*} \simeq \mathrm{Hom}_{F_\pi} ( V_f , F_\pi)(1-k) , & V_{f^*}(k-r) \simeq \mathrm{Hom}_{F_\pi} ( V_f(r) , F_\pi(1)).
}
\]
Then the Euler factor of $V_{f^*}(r)$ at $\ell$ not dividing $p$ is
$$1 - \overline{a_{\ell}(f)} \cdot \ell^{-r}  +  \overline{\psi}  (\ell) \cdot \ell^{k-1 - 2r} .$$
One can compare this with (\ref{eqn:char-poly}) and $\S$\ref{subsec:euler-factor}.


Let $R$ be any $p$-adic ring including $F_\pi$, $\mathcal{O}_{\pi}$, and $\mathcal{O}_{\pi} / \pi^i$. Then, for any $R$-module $M$, we set $M^* \coloneqq \mathrm{Hom}_R(M, R)$.
Also, the torsion part of $M$ is denoted by $M_{\mathrm{tors}}$.

Let $\overline{\rho} : G_\mathbb{Q} \to \mathrm{GL}_2(\mathbb{F})$ be the residual Galois representation of $T_f$.
We assume the following condition throughout this article.
\begin{assu} \label{assu:image}
The image of $\overline{\rho}$ contains $\mathrm{SL}_2(\mathbb{F}_p)$.
\end{assu}
Under Assumption \ref{assu:image} and the $p \geq 5$ assumption, every Galois-stable $\mathcal{O}_\pi$-lattice of $V_f$ is of form $a \cdot T_f$ with $a \in F^\times_\pi$ \cite[pp. 223]{kato-euler-systems}.

\subsection{Iwasawa cohomologies} \label{subsec:iwasawa-cohomologies}
Let $$\Lambda = \mathcal{O}_\pi \llbracket \mathrm{Gal}( \mathbb{Q}(\zeta_{p^\infty}) /\mathbb{Q} ) \rrbracket \simeq \mathcal{O}_\pi [\Delta] \llbracket \Gamma\rrbracket$$
be the Iwasawa algebra where
$\mathbb{Q}(\zeta_{p^\infty}) = \cup_{n \geq 1} \mathbb{Q}(\zeta_{p^n})$ and
$\mathrm{Gal}( \mathbb{Q}(\zeta_{p^\infty}) /\mathbb{Q} ) \simeq \Delta \times \Gamma$
with
$\Delta \simeq \mu_{p-1}$ and $\Gamma \simeq 1 + p\mathbb{Z}_p$.
Let $\mathbb{Q}_{\infty} \coloneqq \mathbb{Q}(\zeta_{p^\infty})^\Delta$ be the cyclotomic $\mathbb{Z}_p$-extension of $\mathbb{Q}$ and $\mathbb{Q}_n$ the unique extension of $\mathbb{Q}$ in $\mathbb{Q}_{\infty}$ of degree $p^n$.
We decompose $$\Lambda = \bigoplus_{i=0}^{p-2} \Lambda_i$$
where $\Lambda_i \coloneqq \Lambda / (\sigma - \omega^i(\sigma) : \sigma \in \Delta)$ and $\omega$ is the Teichm\"{u}ller character. We fix the notational convention for Galois, \'{e}tale, and Iwasawa cohomologies.
\begin{enumerate}
	\item If $K$ is a field, $L$ is a Galois extension of $K$, and $M$ is a topological $\mathrm{Gal}(L/K)$-module, then we denote by
		$$\mathrm{H}^i(L/K, M) \coloneqq \mathrm{H}^i_\mathrm{cont}(\mathrm{Gal}(L/K), M)$$
		the $i$-th Galois cohomology group (\textit{i.e.} the $i$-th continuous cohomology group) of $\mathrm{Gal}(L/K)$ with coefficients in $M$.
		If $L$ is an algebraic closure of $K$, we write $\mathrm{H}^i(K,M)$ instead of $\mathrm{H}^i(L/K, M)$.
	\item If $R$ is a ring and $M$ is an \'{e}tale sheaf on $\mathrm{Spec}(R)$, then we denote by
$$\mathrm{H}^i_{\mathrm{\acute{e}t}}(R, M) \coloneqq \mathrm{H}^i_{\mathrm{\acute{e}t}}(\mathrm{Spec}(R), M)$$
the $i$-th \'{e}tale cohomology group of $\mathrm{Spec}(R)$ with coefficients in $M$.
\item
Let $T$ be a compact $\mathbb{Z}_p$-module with continuous action of $\mathrm{Gal}(\mathbb{Q}_S/\mathbb{Q})$.
We write
\begin{align*}
\mathrm{H}^i_{\mathrm{Iw}}(\mathbb{Q}_S/\mathbb{Q}(\zeta_{p^\infty}), T) & \coloneqq \varprojlim_n \mathrm{H}^i(\mathbb{Q}_S/\mathbb{Q}(\zeta_{p^n}), T), \\
\mathrm{H}^i_{\mathrm{Iw}}(\mathbb{Q}_S/\mathbb{Q}_\infty, T) & \coloneqq \varprojlim_n \mathrm{H}^i(\mathbb{Q}_S/\mathbb{Q}_n, T), \\
\mathrm{H}^i_{\mathrm{Iw}}(\mathbb{Q}_{\infty, \eta}, T) & \coloneqq \varprojlim_n \mathrm{H}^i(\mathbb{Q}_{n, \eta_n}, T), \\
\mathrm{H}^i_{\mathrm{Iw}}(\mathbb{Q}(\zeta_{p^\infty})_\eta, T) & \coloneqq \varprojlim_n \mathrm{H}^i(\mathbb{Q}(\zeta_{p^n})_{\eta_n}, T)
\end{align*}
where $\eta$ is a prime of $\mathbb{Q}_\infty$ or $\mathbb{Q}(\zeta_{p^\infty})$ and $\eta_n$ is the prime of $\mathbb{Q}_{n}$ or $\mathbb{Q}(\zeta_{p^n})$ lying below $\eta$, respectively.
\end{enumerate}

Let
\[
\xymatrix{
j_n : \mathrm{Spec}(\mathbb{Z}[\zeta_{p^n}, S^{-1}]) \to \mathrm{Spec}(\mathbb{Z}[\zeta_{p^n}, 1/p]), &
j : \mathrm{Spec}(\mathbb{Z}[\zeta_{p^\infty}, S^{-1}]) \to \mathrm{Spec}(\mathcal{O}_{\mathbb{Q}(\zeta_{p^\infty})}[1/p])
}
\]
be the natural maps where $S$ is a finite set of places of $\mathbb{Q}$ containing $p$ and the primes where $T$ is ramified. We use the same notations for $\mathbb{Q}_n$ and $\mathbb{Q}_{\infty}$.
\begin{defn}
Let $a = 1$ or $2$, and $i \in \lbrace 0, 1, \cdots, p-3, p-2 \rbrace$.
\begin{enumerate}
\item
We define the \textbf{$a$-th Iwasawa cohomology for $T_{f}(k-r)$ over $\Lambda$} by
$$\mathbf{H}^a( j_*  T_{f}(k-r)) \coloneqq \varprojlim_{n} \mathrm{H}^a_{\mathrm{\acute{e}t}}( \mathrm{Spec}(\mathbb{Z}[\zeta_{p^n}, 1/p]), j_{n,*}T_{f}(k-r)) $$
where $\mathrm{H}^a_{\mathrm{\acute{e}t}}( \mathrm{Spec}(\mathbb{Z}[ \zeta_{p^n}, 1/p]),j_{n,*}T_{f}(k-r))$ is the \'{e}tale cohomology group.
\item
We define the \textbf{$a$-th Iwasawa cohomology for $T_{f, i}(k-r)$ over $\Lambda_i$} by
$$\mathbb{H}^a( j_*  T_{f, i}(k-r)) \coloneqq \mathbf{H}^a( j_*  T_{f}(k-r))^{\omega^i} $$
where $T_{f, i} = T_f \otimes \omega^i$ and
$\mathbf{H}^a( j_*  T_{f}(k-r))^{\omega^i}$ is the $\omega^i$-isotypic component of $\mathbf{H}^a( j_*  T_{f}(k-r))$.
\end{enumerate}
\end{defn}
\begin{rem}
We use $\mathbf{H}^a$ for the extension $\mathbb{Q}(\zeta_{p^\infty})/\mathbb{Q}$ and $\mathbb{H}^a$ for the extension $\mathbb{Q}_\infty/\mathbb{Q}$.
\end{rem}
\begin{thm}[Kato] \label{thm:kato-iwasawa-cohomologies}
For all $i = 0, \cdots , p-2$,
\begin{enumerate}
\item $\mathbb{H}^1(j_*V_{f,i}(k-r))$ is free of rank one over $\Lambda_i \otimes \mathbb{Q}_p$;
\item $\mathbb{H}^1(j_*T_{f,i}(k-r))$ is free of rank one over $\Lambda_i$ under Assumption \ref{assu:image};
\item $\mathbb{H}^2(j_*T_{f,i}(k-r))$ is a finitely generated torsion module over $\Lambda_i$.
\end{enumerate}
\end{thm}
\begin{proof}
See \cite[Theorem 12.4.(1) and (2)]{kato-euler-systems}, where it is proved for the module $\mathbf{H}^i(j_*T_{f})$. Note that the statement is insensitive to Tate twists.
\end{proof}

\subsection{Kato's zeta elements and the main conjecture}
We keep Assumption \ref{assu:image} in this subsection to ensure the integrality of Kato's zeta elements.
Let
$$Z(f, T_f(k-r))  \subseteq \mathbf{H}^1(j_*T_{f}(k-r))$$
 be the $\Lambda$-module of Kato's zeta elements generated by
$\mathbf{z}^{(p)}_\gamma \otimes (\zeta^{\otimes k-r}_{p^n})_n$ where $\gamma$ runs over $T_f$.
Here, $\mathbf{z}^{(p)}_\gamma$ appears in \cite[Theorem 12.5]{kato-euler-systems}.
More precisely, there exists a map $T_f \to \mathbf{H}^1(j_*T_f)$ defined by $\gamma \mapsto \mathbf{z}^{(p)}_\gamma$.
\begin{prop}
Under Assumption \ref{assu:image}, $Z(f, T_f(k-r))$ is generated by one element over $\Lambda$.
\end{prop}
\begin{proof}
It is a well-known fact. See \cite[Proposition A.12]{kim-nakamura}, for example.
\end{proof}
\begin{defn} \label{defn:kato-zeta-elements}
We define \textbf{the zeta element $\mathbf{z}_{\mathrm{Kato}}(f, k-r)$} to be a generator of $Z(f, T_f(k-r))$, and define $\mathbf{z}_{\mathrm{Kato}}(f,i, k-r)$ by the $\omega^i$-isotypic component of $\mathbf{z}_{\mathrm{Kato}}(f, k-r)$.
Then $\mathbf{z}_{\mathrm{Kato}}(f,i, k-r)$ generates $Z(f, T_{f,i}(k-r)) \subseteq \mathbb{H}^1(j_*T_{f,i}(k-r))$, the $\omega^i$-component of $Z(f, T_f(k-r))$.
\end{defn}
\begin{rem}
Via the dual exponential map, $\mathbf{z}_{\mathrm{Kato}}(f, i, k-r)$ interpolates the $L$-values of $f^*$ at $s = r$ twisted by $\omega^{-i}\chi^{-1}$ where $\chi$ runs over finite order characters of $\Gamma$. See \cite[Theorem 12.5]{kato-euler-systems}.
\end{rem}
\begin{conj}[Kato's main conjecture] \label{conj:kato-main-conjecture}
Let $T_f$ be an $\mathcal{O}_\pi$-lattice of $V_f$ and $1 \leq r \leq k-1$.
Keep Assumption \ref{assu:image}.
The following equivalent statements hold.
\begin{enumerate}
\item
Let $\mathfrak{p}$ be a height one prime ideal of $\Lambda$.
Then we have
$$\mathrm{length}_{\Lambda_{\mathfrak{p}}} \mathbf{H}^1( j_* T_f(k-r) )_{\mathfrak{p}} / Z(f, T_f(k-r) )_{\mathfrak{p}}
=
\mathrm{length}_{\Lambda_{\mathfrak{p}}} \mathbf{H}^2( j_* T_f(k-r) )_{\mathfrak{p}}$$
\item
Let $\mathfrak{p}_i$ be a height one prime ideal of $\Lambda_i$.
Then we have
$$\mathrm{length}_{\Lambda_{i, \mathfrak{p}_i}} \mathbb{H}^1( j_* T_{f, i}(k-r) )_{\mathfrak{p}_i} / Z(f, T_{f, i}(k-r) )_{\mathfrak{p}_i}
=
\mathrm{length}_{\Lambda_{i, \mathfrak{p}_i}} \mathbb{H}^2( j_* T_{f, i}(k-r) )_{\mathfrak{p}_i}$$
for all $i = 0, \cdots , p-2$.
\item
We have
$$\mathrm{char}_{\Lambda_i} \mathbb{H}^1(j_*T_{f,i}(k-r)) / Z(f, T_{f, i}(k-r) )
= \mathrm{char}_{\Lambda_i} \mathbb{H}^2(j_*T_{f,i}(k-r)) $$
for all $i = 0, \cdots , p-2$.
\end{enumerate}
\end{conj}
\begin{rem}
See \cite[Conjecture 12.10]{kato-euler-systems} for the first statement (with ``$r=0$'') and \cite[Conjecture 6.1]{kurihara-invent} for the third statement (with $i=0$).
Note that Conjecture \ref{conj:kato-main-conjecture} is independent of $r$.
The decomposition using powers of the Teichm\"{u}ller character is required to consider Iwasawa invariants as in \cite{epw}.
\end{rem}
The following one-sided divisibility is proved in \cite[Theorem 12.5.(4)]{kato-euler-systems}.
\begin{thm}[Kato] \label{thm:kato-divisibility}
Keep Assumption \ref{assu:image}. Then the inclusion $\subseteq$ holds in each statement of Conjecture \ref{conj:kato-main-conjecture}.
\end{thm}
We say that the \textbf{$\omega^i$-component of Kato's main conjecture holds} if the equality holds for $i$ in the second or equivalently the third statement in Conjecture \ref{conj:kato-main-conjecture}.

For a finitely generated torsion $\Lambda_i$-module $M_i$, we denote by $\mathrm{char}_{\Lambda_i} (M_i)$ the associated characteristic ideal, and by $\mu(M_i) = \mu(\mathrm{char}_{\Lambda_i} (M_i) )$ and $\lambda(M_i) = \lambda(\mathrm{char}_{\Lambda_i} (M_i) )$, the Iwasawa invariants of $M_i$.
If two $\Lambda_i$-ideals have the same $\mu$-invariants and the same $\lambda$-invariants and furthermore one divides the other, they must be equal. We hence find the following.
\begin{cor} \label{cor:reduction-to-iwasawa-invariants}
If the Iwasawa invariants of the $\omega^i$-components of both sides in Conjecture \ref{conj:kato-main-conjecture} coincide, then the $\omega^i$-component of Kato's main conjecture holds.
\end{cor}

\section{Main results and applications} \label{sec:main-results}
\subsection{The statement of main theorem}
Let $\overline{\rho}$ be a mod $p$ residual representation with conductor $N(\overline{\rho})$.
Denote by $S_k(\overline{\rho})$ the set of newforms of fixed weight $k$ such that the residual representation is isomorphic to $\overline{\rho}$ and $p$ does not divide the levels of the newforms.
\begin{thm} \label{thm:main-theorem}
Assume that
\begin{itemize}
\item
$2 \leq k \leq p-1$;
\item
the image of $\overline{\rho}$ contains $\mathrm{SL}_2(\mathbb{F}_p)$;
\item $\mu \left( \dfrac{ \mathbb{H}^1(j_*T_{f_0, i}(k-r)) }{ \mathbf{z}_{\mathrm{Kato}}(f_0, i, k-r) } \right) = 0$ for one $f_0 \in S_k(\overline{\rho})$.
\end{itemize}
Then we have the following statements:
\begin{enumerate}
\item
$$\mu \left( \dfrac{ \mathbb{H}^1(j_*T_{f, i}(k-r))}{ \mathbf{z}_{\mathrm{Kato}}(f, i, k-r) } \right) = \mu(\mathbb{H}^2(j_*T_{f, i}(k-r))) = 0$$
for all $f \in S_k(\overline{\rho})$.
\item
	\[
		\lambda \left( \dfrac{ \mathbb{H}^1(j_*T_{f, i}(k-r))}{ \mathbf{z}_{\mathrm{Kato}}(f, i, k-r) } \right) - \lambda \left( \mathbb{H}^2( j_*T_{f, i}(k-r) ) \right)
	\]
is constant for $f \in S_k(\overline{\rho})$.
\item
If the $\omega^i$-component of Kato's main conjecture holds for one form in $S_k(\overline{\rho})$, then the $\omega^i$-component of Kato's main conjecture holds for all forms in $S_k(\overline{\rho})$.
\end{enumerate}
\end{thm}
\begin{proof}
\begin{enumerate}
\item It is proved in Corollary \ref{cor:mu-invariants-vanish-in-families} with Kato's divisibility statement (Theorem \ref{thm:kato-divisibility}).
\item It is proved in $\S$\ref{sec:connection}.
\item It immediately follows from the above two statements due to Corollary \ref{cor:reduction-to-iwasawa-invariants}.
\end{enumerate}
\end{proof}
\begin{rem}
\begin{enumerate}
\item
The assumption $\mu \left( \dfrac{ \mathbb{H}^1(j_*T_{f_0, i}(k-r)) }{ \mathbf{z}_{\mathrm{Kato}}(f_0, i, k-r) } \right) = 0$ in Theorem \ref{thm:main-theorem} is \emph{weaker} than the $\mu=0$ assumptions of various $p$-adic $L$-functions in other literatures (e.g. Corollary \ref{cor:p-adic-l-functions}) since those $p$-adic $L$-functions are the image of Kato's zeta elements under the composition of the localization map at $p$ and the corresponding Coleman maps. We do not know whether this composition of two maps preserves the mod $p$ non-vanishing in general.
\item In Theorem \ref{thm:main-theorem}.(1), we observe in particular that, if the $\mu=0$ conjecture for the fine Selmer group holds for one form in $S_k(\overline{\rho})$, then it holds for all forms in $S_k(\overline{\rho})$. The $\mu=0$ conjecture for fine Selmer groups  is due to Coates--Sujatha \cite[Conjecture A]{coates-sujatha-fine-selmer}. Indeed, the propagation of the $\mu=0$ conjecture for fine Selmer groups can also be obtained by using Proposition \ref{prop:iwasawa-invariants-H^2}, Proposition \ref{prop:iwasawa-invariants-H^2-imprimitive}, and (\ref{eqn:mod-p-H2}) only.
\item Theorem \ref{thm:main-theorem}.(2) is weaker than the $\lambda$-invariant formulas in the literature.
This is due to the fact that it is unclear whether the Iwasawa modules we deal with have no finite Iwasawa submodule.
\item One may consider the congruences between modular forms of different weights. It is possible in the ordinary case via Hida theory \cite{epw}. 
\item  Under Assumption \ref{assu:image}, we have
$$\dfrac{ \mathbb{H}^1(j_*T_{f, i}(k-r))}{ \mathbf{z}_{\mathrm{Kato}}(f, i, k-r) } \simeq \dfrac{\Lambda_i }{g_i(T)} $$
for some $g_i(T) \in \Lambda_i$ due to \cite[Theorem 12.4.(3)]{kato-euler-systems}; thus, the information of Iwasawa invariants is encoded in $g_i(T)$.
\end{enumerate}
\end{rem}
In the semi-stable ordinary case, we can remove the Fontaine--Laffaille assumption as follows.
Since we do not use any Hida deformation explicitly, our result is weaker than that of \cite{epw} but the argument is much simpler than  theirs. For example, we do not need to construct two-variable $p$-adic $L$-functions.

Denote by $S^{\mathrm{ord}}_k(\overline{\rho})$ the set of $p$-ordinary newforms of \emph{fixed} weight $k$ and level $N$ or $Np$ such that the residual representation is isomorphic to $\overline{\rho}$. Here the $p$-ordinarity means that the $p$-th Fourier coefficient of the newform is a $p$-adic unit.
\begin{thm} \label{thm:main-theorem-ordinary}
Assume that
\begin{itemize}
\item the image of $\overline{\rho}$ contains $\mathrm{SL}_2(\mathbb{F}_p)$;
\item $\overline{\rho}$ is ordinary at $p$;
\item  $\overline{\rho}$ has distinct Jordan--H\"{o}lder factors on the decomposition group at $p$;
\item $\mu \left( \dfrac{ \mathbb{H}^1(j_*T_{f_0, i}(k-r)) }{ \mathbf{z}_{\mathrm{Kato}}(f_0, i, k-r) } \right) = 0$ for one $f_0 \in S^{\mathrm{ord}}_k(\overline{\rho})$.
\end{itemize}
Then the same conclusions of Theorem \ref{thm:main-theorem} holds for $S^{\mathrm{ord}}_k(\overline{\rho})$.
\end{thm}
\begin{proof}
The key observation is that the Fontaine--Laffaille condition ($2 \leq k \leq p-1$ and $p \nmid N$)  is required only when we invoke the mod $p$ multiplicity one (Theorem \ref{thm:mod-p-multiplicity-one}) and Ihara's lemma (Theorem \ref{thm:ihara-lemma}).
For the mod $p$ multiplicity one, see \cite[Theorem 1.13]{vatsal-duke} which uses \cite[Theorem 2.1]{wiles} when $k = 2$ and \cite{hida-invent-1986}  when $k > 2$. The equivalent statement can be found in \cite[Propositions 3.1.1 and 3.3.1]{epw}.
See also \cite[$\S$3.8]{epw} for how Ihara's lemma is used in the integral period comparison (cf. \cite[Proposition 4.5]{vatsal-integralperiods-2013}).
In $\S$\ref{sec:zeta-elements}, we explain how the integral period comparison yields the congruence between Kato's zeta elements.
\end{proof}
\begin{rem}
Nakamura has recently constructed Kato's Euler system over the universal deformation space by using the $p$-adic local Langlands correspondence and the local-global compatibility result.
Combining his construction with the technique in this paper, the Fontaine--Laffaille condition in Theorem \ref{thm:main-theorem} could be entirely removed. See \cite{nakamura-kato-deformation} for details.
\end{rem}

\subsection{Applications to main conjectures with $L$-functions}
We describe the relation with other main conjectures focusing more on the non-ordinary case. We do not recall the formulations of the Greenberg-style, $\pm$-, and $\sharp/\flat$-main conjectures in this article, but one can find details in \cite{kato-euler-systems}, \cite{kobayashi-thesis}, and \cite{sprung-ap-nonzero} for modular forms of weight two.

When $p$ divides $a_p(f)$ and the weight of $f$ is two, let
$$\widetilde{L}^{\bullet}_p(f^*, -i) =
\left\lbrace \begin{array}{ll}
L^{\bullet}_p(f^*, -i) & \textrm{if } i = 0 \\
\dfrac{1}{\gamma-1} \cdot L^{\bullet}_p(f^*, -i) &  \textrm{if } i \neq 0 \textrm{ and } \bullet = -
\end{array} \right.
$$
where $L^{\bullet}_p(f^*, -i)$ is the $\omega^{-i}$-component of relevant integral $p$-adic $L$-functions, $\bullet \in \lbrace +, - , \sharp, \flat \rbrace$, and $\gamma$ is a generator of $\Lambda_i$. It is known that $\widetilde{L}^{\bullet}_p(f^*, -i) \in \Lambda_i$ under Assumption \ref{assu:image}. Note that we follow Kobayashi's convention for $\pm$.
Let $\mathrm{Im} \left( \mathrm{Col}^{\sharp/\flat, i} \right) \subseteq \Lambda_i$ be the image of the $\sharp/\flat$-Coleman maps defined by
 $\mathrm{Col}^{\flat, i} = (\underline{\mathrm{Col}}_1)^{\omega^i}$ and
 $\mathrm{Col}^{\sharp, i} = (\underline{\mathrm{Col}}_2)^{\omega^i}$, respectively, following  \cite{lei-loeffler-zerbes_wach}.
\begin{thm}[Kato, Kobayashi, Lei, Sprung, Lei--Loeffler--Zerbes] \label{thm:equivalent-main-conjecture}
Keep Assumption \ref{assu:image}.
The $\omega^i$-component of Kato's main conjecture (with $r=1$) is equivalent to:
\begin{enumerate}
\item $\left(  L_p(f^*, -i) \right)  = \mathrm{char}_{\Lambda_i}\mathrm{Sel}(\mathbb{Q}_\infty, A_{f^*, -i}(1))^\vee$ when $a_p(f)$ is a $\pi$-adic unit, or
\item $\left( \widetilde{L}^{\pm}_p(f^*, -i) \right) = \mathrm{char}_{\Lambda_i} \mathrm{Sel}^{\pm}(\mathbb{Q}_\infty, A_{f^*, -i}(1))^\vee$ when $a_p(f) = 0$ and the weight is two, or
\item $\mathrm{char}_{\Lambda_i} \left( \mathrm{Im} \left( \mathrm{Col}^{\sharp/\flat, i} \right) / L^{\sharp/\flat}_p(f^*, -i) \right)  = \mathrm{char}_{\Lambda_i} \mathrm{Sel}^{\sharp/\flat}(\mathbb{Q}_\infty, A_{f^*, -i}(1))^\vee$ when $a_p(f)$ is divisible by $\pi$
\end{enumerate}
where $A_{f^*, -i} \coloneqq V_{f^*, -i} / T_{f^*, -i}$.
\end{thm}
\begin{proof}
This is the combination of \cite[$\S$17.13]{kato-euler-systems}, \cite[Theorem 7.4]{kobayashi-thesis}, \cite[Corollary 6.8]{lei-compositio}, \cite[Proposition 7.19]{sprung-ap-nonzero}, and \cite[Corollary 6.6]{lei-loeffler-zerbes_wach}.
\end{proof}
The following corollary is immediate from Theorem \ref{thm:main-theorem} and Theorem \ref{thm:equivalent-main-conjecture}.
\begin{cor} \label{cor:p-adic-l-functions}
We keep all the assumptions in Theorem \ref{thm:main-theorem}. Let $f \in S_k(\overline{\rho})$.
\begin{enumerate}
\item Assume that $f$ is ordinary at $p$.
 If
$$ \left( L_p(f^*, -i) \right) = \mathrm{char}_{\Lambda_i}\mathrm{Sel}(\mathbb{Q}_\infty, A_{f^*, -i}(1))^\vee,$$
then
$$ \left( L_p(g^*, -i) \right) = \mathrm{char}_{\Lambda_i}\mathrm{Sel}(\mathbb{Q}_\infty, A_{g^*, -i}(1))^\vee$$
 for all $g \in S_k(\overline{\rho})$.
\item Assume that $a_p(f) = 0$ and $k=2$.
 If
$$ \left( \widetilde{L}^{\pm}_p(f^*, -i) \right) = \mathrm{char}_{\Lambda_i}\mathrm{Sel}^{\pm}(\mathbb{Q}_\infty, A_{f^*, -i}(1))^\vee,$$
then
$$ \left( \widetilde{L}^{\pm}_p(g^*, -i) \right) = \mathrm{char}_{\Lambda_i}\mathrm{Sel}^{\pm}(\mathbb{Q}_\infty, A_{g^*, -i}(1))^\vee$$
 for all $g \in S_k(\overline{\rho})$.
\item  Assume that $p \mid a_p(f)$.
 If
$$ \mathrm{char}_{\Lambda_i} \left( \mathrm{Im} \left( \mathrm{Col}^{\sharp/\flat, i} \right) / L^{\sharp/\flat}_p(f^*, -i) \right) = \mathrm{char}_{\Lambda_i}\mathrm{Sel}^{\sharp/\flat}(\mathbb{Q}_\infty, A_{f^*, -i}(1))^\vee,$$
 then
$$ \mathrm{char}_{\Lambda_i} \left( \mathrm{Im} \left( \mathrm{Col}^{\sharp/\flat, i} \right) / L^{\sharp/\flat}_p(g^*, -i) \right) = \mathrm{char}_{\Lambda_i}\mathrm{Sel}^{\sharp/\flat}(\mathbb{Q}_\infty, A_{g^*, -i}(1))^\vee$$
 for all $g \in S_k(\overline{\rho})$.
\end{enumerate}
\end{cor}
\begin{rem}
The mod $p$ non-vanishing of Kato's zeta elements is weaker than vanishing of $\mu$-invariants of all the above $p$-adic $L$-functions as mentioned before.
\end{rem}
\begin{rem}
In the work of Skinner--Urban \cite{skinner-urban}, the ($\omega^0$-component of the) main conjecture for a large class of elliptic curves and modular forms at good ordinary primes is proved; more precisely, the main conjecture is proved under the following conditions:
\begin{itemize}
\item $f$ is good ordinary at $p$;
\item $k \equiv 2 \pmod{p-1}$, so only $k=2$ is allowed in our setting;
\item $\psi = \mathbf{1}$, so the Nebentypus is trivial;
\item $r = k-1$;
\item $i=0$, so $\mathbb{Q}_{\infty}/\mathbb{Q}$ is only considered;
\item there exists a prime $\ell \neq p$ such that $ \ell \Vert N(\overline{\rho})$.
\end{itemize}
The second and third conditions come from the anticyclotomic input, the work of Vatsal on vanishing of anticyclotomic $\mu$-invariants \cite{vatsal-duke}. The second condition on weight would be removed by the work of Chida--Hsieh \cite[Remark 1 after Theorem C]{chida-hsieh-p-adic-L-functions}. 
The fourth condition follows from the convention of the Galois modules in the Selmer groups in \cite[$\S$1.1]{skinner-urban}.
Note that our main theorem (Theorem \ref{thm:main-theorem}) applies without these assumptions Skinner--Urban made.
\end{rem}

\section{``Prime-to-$p$ local'' Iwasawa theory} \label{sec:local}
We quickly review \cite[Proposition 2.4]{gv} under our setting.
\subsection{Euler factors} \label{subsec:euler-factor}
We recall several Euler factors following \cite[Propositions 8.7, 8.10 and 8.12]{kato-euler-systems}.
Let $\ell$ be a prime different from $p$.
\begin{defn}
\begin{enumerate}
\item At the level of $\mathrm{H}^1(\mathbb{Z}[1/p], j_* V_{k, \mathbb{Z}_p}(Y(p^nN))(k-r))$, we define
$$E_\ell(Y(p^nN), r) \coloneqq \left\lbrace
\begin{array}{ll}
 1 - T'(\ell) \cdot \begin{pmatrix}
1/\ell & 0 \\ 0 & 1
\end{pmatrix}^* \cdot \ell^{-r} + \begin{pmatrix}
1/\ell & 0 \\ 0 & 1/\ell
\end{pmatrix}^* \cdot \ell^{k-1-2r}  & \textrm{ if } \ell \nmid N \\
 1 - T'(\ell) \cdot \begin{pmatrix}
1/\ell & 0 \\ 0 & 1
\end{pmatrix}^* \cdot \ell^{-r}  & \textrm{ if } \ell \mid N
\end{array} \right.$$
where $T'(\ell)$ is the dual Hecke operator at $\ell$ as defined in \cite[$\S$4.9]{kato-euler-systems}.
\item At the level of $\mathrm{H}^1(\mathbb{Z}[\zeta_{p^n}, 1/p], j_* V_{k, \mathbb{Z}_p}(Y_1(N))(k-r))$, we define
$$E_\ell(Y_1(N), p^n, r) \coloneqq 1 - T'(\ell) \cdot \sigma^{-1}_\ell \cdot \ell^{-r} + \Delta'(\ell) \cdot \sigma^{-2}_\ell \cdot \ell^{k-1-2r} $$
where
$\sigma_\ell$ is the arithmetic Frobenius at $\ell$ in $\mathrm{Gal}(\mathbb{Q}(\zeta_{p^n})/\mathbb{Q})$,
$\Delta'(\ell) = \begin{pmatrix}
\ell & 0 \\ 0 & 1/\ell
\end{pmatrix}^*$ if $\ell \nmid N$ and $\Delta'(\ell) = 0$ otherwise.
\item At the level of $\mathrm{H}^1(\mathbb{Z}[\zeta_{p^n}, 1/p], j_* T_{f}(k-r))$, we define
$$E_\ell(f^*, p^n, r) \coloneqq  1 - \overline{a_\ell(f)} \cdot \sigma^{-1}_\ell \cdot \ell^{-r} + \overline{\psi}(\ell) \cdot \sigma^{-2}_\ell \cdot \ell^{k-1-2r}  .$$
where $\sigma_\ell$ is the arithmetic Frobenius element in $\mathrm{Gal}(\mathbb{Q}(\zeta_{p^n})/\mathbb{Q})$.
\item At the level of Iwasawa cohomologies, we define
$$\mathcal{E}_\ell(f^*, r) \coloneqq E_\ell(f^*, p^\infty, r)=  1 - \overline{a_\ell(f)} \cdot \sigma^{-1}_\ell \cdot \ell^{-r} + \overline{\psi}(\ell) \cdot \sigma^{-2}_\ell \cdot \ell^{k-1-2r}  \in \Lambda$$
and
$$\mathcal{E}_\ell(f^*, -i, r) \coloneqq 1 - \overline{a_\ell(f)} \cdot \omega^{-i} \left( \sigma^{-1}_\ell \right) \cdot \langle \sigma^{-1}_\ell \rangle \cdot \ell^{-r} + \overline{\psi}(\ell) \cdot \omega^{-i} \left( \sigma^{-2}_\ell \right) \cdot \langle \sigma^{-2}_\ell \rangle \cdot \ell^{k-1-2r} \in \Lambda_{-i} $$
where $\sigma_\ell$ is the arithmetic Frobenius element in $\mathrm{Gal}(\mathbb{Q}(\zeta_{p^\infty})/\mathbb{Q})$ and
$\langle - \rangle : \mathbb{Z}^\times_p \to 1 + p\mathbb{Z}_p$ is the projection to 1-units.
\end{enumerate}
\end{defn}
\subsection{Local cohomologies} \label{subsec:local-cohomologies}
Let $\ell$ be a prime different from $p$ and $\eta$ be a prime of $\mathbb{Q}_\infty$ lying above $\ell$.
Then we have the following statement.
\begin{prop} \label{prop:gv-local-H^0}
The ideal $\mathrm{char}_{\Lambda_{-i}} \left( \bigoplus_{\eta \mid \ell} \mathrm{H}^0(\mathbb{Q}_{\infty, \eta}, \mathrm{Hom}(T_{f^*, -i}(r), ( F_\pi /\mathcal{O}_\pi ) (1)  ))^{ \vee} \right)$ is generated by $\mathcal{E}_\ell(f^*, -i, r)$ over $\Lambda_{-i}$ where $(-)^\vee$ means the Pontryagin dual.
\end{prop}
\begin{proof}
See \cite[Proposition 2.4]{gv}.
By \cite[Proposition 2]{greenberg-general-iwasawa}, it is known that
\[
\mathrm{H}^1(\mathbb{Q}_{\infty, \eta}, A_{f^*,-i}(r))^{\iota, \vee} \sim \mathrm{H}^0(\mathbb{Q}_{\infty, \eta}, \mathrm{Hom}_{F_\pi}( T_{f^*,-i}(r),( F_\pi /\mathcal{O}_\pi ) (1) ) )^{\vee}
\]
where $\sim$ means a pseudo-isomorphism over $\Lambda_{-i}$.
\end{proof}
\begin{defn} \label{defn:involution}
The involution map $\iota : \Lambda \to \Lambda$ is defined by $\gamma \mapsto \gamma^{-1}$ where $\gamma \in \mathrm{Gal}(\mathbb{Q}(\zeta_{p^\infty})/\mathbb{Q})$.
For a $\Lambda$-module $M$, $M^\iota$ is defined by the same underlying module with the inverse $\Lambda$-action via $\iota$. Especially, if $M$ is a $\Lambda_{i}$-module, then $M^\iota$ is a $\Lambda_{-i}$-module.
Denote by $f^{\iota} \in \Lambda_i$ the image of $f \in \Lambda_{-i}$ under $\iota$.
\end{defn}
Combining with the local Tate duality, Proposition \ref{prop:gv-local-H^0} implies the following statement.
\begin{cor} \label{cor:local-H^2}
The ideal $\mathrm{char}_{\Lambda_{-i}} \left( \bigoplus_{\eta \mid \ell} \mathrm{H}^2_{\mathrm{Iw}}(\mathbb{Q}_{\infty, \eta}, T_{f^*, -i}(r)) \right)$ is generated by $\mathcal{E}_\ell(f^*, -i, r)$ over $\Lambda_{-i}$ with $\mu = 0$ and $\lambda = s_\ell \cdot d_\ell$
where $s_\ell$ is the number of primes of $\mathbb{Q}_\infty$ lying above $\ell$
and
$d_{\ell} = \dim_{F_\pi}\mathrm{H}^1(\mathbb{Q}_{\infty, \eta}, V_{f, i}(k-r))$.
\end{cor}
Applying the same argument to the dual representation, the following statement follows.
\begin{cor} \label{cor:local-H^2-reversed}
The ideal $\mathrm{char}_{\Lambda_i} \left( \bigoplus_{\eta \mid \ell} \mathrm{H}^2_{\mathrm{Iw}}(\mathbb{Q}_{\infty, \eta}, T_{f, i}(k-r)) \right)$ is generated by $\mathcal{E}_\ell(f, i, k-r) = \mathcal{E}_\ell(f^*, -i, r)^\iota$ over $\Lambda_i$ with $\mu = 0$ and $\lambda = s_\ell \cdot d_\ell$.
\end{cor}

\section{The zeta element side} \label{sec:zeta-elements}
We identify
$$ V_{k,\mathbb{Z}_p}(Y_1(N)) = \mathrm{H}^1_{\mathrm{\acute{e}t}}(Y_1(N)_{\overline{\mathbb{Q}}}, \mathrm{Sym}^{k-2}_{\mathbb{Z}_p}(\mathcal{H}^1_p)) \simeq \mathrm{H}^1(Y_1(N)(\mathbb{C}), \mathrm{Sym}^{k-2}(\mathbb{Z}^2_p) ) $$
and the former naturally admits the action of $\mathrm{Gal}(\overline{\mathbb{Q}}/\mathbb{Q})$. See \cite[Thm. 1.1]{faltings-jordan} for the properties of this Galois representation.
The complex conjugation acts on the latter and it also induces the action on $T_f$.
Denote by $T^{\pm}_f$ the part on which the complex conjugation acts by $\pm 1$, respectively. The same rule applies to $V_f$.
\subsection{Hecke algebras and Galois representations} \label{subsec:hecke-algebras}
We denote by $\mathbb{T}(N)$
the $\mathbb{Z}_p$-subalgebra of $\mathrm{End}_{\mathbb{Z}_p} \left( V_{k,\mathbb{Z}_p}(Y_1(N)) \right)$
generated by all the Hecke operators $T_r \ (r \nmid N)$, $U_r \ (r \mid N)$, and $\langle d \rangle \  (d \in (\mathbb{Z}/N\mathbb{Z})^\times)$. This is equivalent to $\mathbb{T}$ in \cite[\S2]{faltings-jordan}.
For a maximal ideal $\mathfrak{m} \subseteq \mathbb{T}(N)$, write $\mathbb{T}(N)_\mathfrak{m}$ to be the localization of $\mathbb{T}(N)$ at $\mathfrak{m}$.
We also denote by $M_\mathfrak{m}$ the localization of $\mathbb{T}(N)$-module $M$ at $\mathfrak{m}$.

Let $S$ be a finite set of places of $\mathbb{Q}$ containing the primes dividing $pN$, and let $\mathbb{Q}_S$ be the maximal extension of $\mathbb{Q}$ unramified outside $S$.
Thanks to the work of Deligne \cite{deligne-modular-galois} and its subsequent work including \cite[Proposition 11.1]{gross-tameness},  \cite[Proposition 5.1]{inv100}, and \cite[\S2]{faltings-jordan},
there exists the residual Galois representation
$$\overline{\rho} : \mathrm{Gal}(\mathbb{Q}_S/\mathbb{Q}) \to \mathrm{GL}_2(\mathbb{T}(N) / \mathfrak{m}\mathbb{T}(N)) \simeq \mathrm{GL}_2(\mathbb{F})$$
characterized by
$$\mathrm{Tr}(\overline{\rho} (\mathrm{Frob}_q)) = T_q \pmod{ \mathfrak{m}}$$
for all primes $q$ not dividing $Np$.
\begin{defn}
We say that \textbf{$\mathfrak{m}$ is non-Eisenstein} if the corresponding residual representation $\overline{\rho}$ is irreducible.
\end{defn}
By localizing at $\mathfrak{m}$, we obtain Galois representation $V_{k,\mathbb{Z}_p}(Y_1(N))_{\mathfrak{m}} = \varprojlim_n (V_{k,\mathbb{Z}_p}(Y_1(N))/\mathfrak{m}^n )$ over the localized Hecke algebra $\mathbb{T}(N)_\mathfrak{m} = \varprojlim_n (\mathbb{T}(N)/\mathfrak{m}^n)$.
See \cite{carayol-modular-galois-repns} for the description of this Galois representation and also \cite[page 11]{faltings-jordan}).

We will see that $V_{k,\mathbb{Z}_p}(Y_1(N))_{\mathfrak{m}} / \mathfrak{m} = V_{k,\mathbb{Z}_p}(Y_1(N)) / \mathfrak{m} \simeq \overline{\rho}$ (Corollary \ref{cor:mod-p-multiplicity-one}).
In the non-Eisenstein case, it is known that $\mathbb{T}(N)_\mathfrak{m}$ is reduced.
This is essentially due to Wiles \cite[Proposition 2.15]{wiles} (cf. \cite[Proposition 2.4.2]{epw}) and the statement for our (indeed, a more general) setting can be found in \cite[Lemma 5.4.(iii)]{dimitrov-ihara}.
\subsection{$\Sigma_0$-imprimitive zeta elements}
We pin down $\Sigma_0 = S'  := S \setminus \lbrace p, \infty \rbrace$.
Indeed, we have
\begin{equation} \label{eqn:zeta-elements}
\mathbf{z}_{\mathrm{Kato}}(f, i, k- r) = \left( \mathbf{z}^{(p)}_{\gamma_f} \otimes (\zeta_{p^n})^{\otimes k-r}_n \right)^{\omega^i} \in \mathbb{H}^1(j_*T_{f,i}(k-r))
\end{equation}
where $\gamma_f = \gamma^+_f + \gamma^-_f \in T_f$ satisfies $\gamma^{\pm}_f \in T^{\pm}_f$ and $T_f = \mathcal{O}_\pi \gamma^+_f + \mathcal{O}_\pi \gamma^-_f $.
This is an explicit description of $\mathbf{z}_{\mathrm{Kato}}(f, i, k- r)$ (cf. Definition \ref{defn:kato-zeta-elements}).

We define the \textbf{$\Sigma_0$-imprimitive zeta element of $\mathbf{z}_{\mathrm{Kato}}(f, i, k- r)$} by
\begin{align*}
\mathbf{z}^{\Sigma_0}_{\mathrm{Kato}}(f, i, k- r) & \coloneqq \left( \prod_{\ell \in \Sigma_0} \mathcal{E}_\ell(f^*, -i, r)^\iota \right) \cdot \mathbf{z}_{\mathrm{Kato}}(f, i, k- r)
\end{align*}
where $\mathcal{E}_\ell(f^*, -i, r)^\iota$ is the image of $\mathcal{E}_\ell(f^*, -i, r)$ under $\iota : \Lambda_{-i} \to \Lambda_{i}$ in Definition \ref{defn:involution}.
For notational convenience, we write
\[
\xymatrix{
\mu \left( \mathbf{z} \right)  =  \mu \left( \mathbb{H}^1(j_*T_{f,i}(k-r)) /  \mathbf{z} \right)  , & \lambda \left( \mathbf{z} \right)  = \lambda \left( \mathbb{H}^1(j_*T_{f,i}(k-r)) / \mathbf{z} \right)
}
\]
for $\mathbf{z} \in \mathbb{H}^1(j_*T_{f,i}(k-r))$.
Then we have
\begin{align} \label{eqn:iwasawa-invariants-kato-elements}
\begin{split}
\mu \left( \mathbf{z}^{\Sigma_0}_{\mathrm{Kato}}(f, i, k- r) \right) & = \mu \left( \mathbf{z}_{\mathrm{Kato}}(f, i, k- r) \right) , \\
\lambda \left( \mathbf{z}^{\Sigma_0}_{\mathrm{Kato}}(f, i, k- r) \right) & =
\lambda \left( \mathbf{z}_{\mathrm{Kato}}(f, i, k- r) \right)
+ \sum_{\ell \in \Sigma_0}   \lambda \left( \bigoplus_{\eta \mid \ell} \mathrm{H}^2_{\mathrm{Iw}}(\mathbb{Q}_{\infty, \eta}, T_{f, i} (k-r) )  \right)
\end{split}
\end{align}
by Corollary \ref{cor:local-H^2-reversed}.

\subsection{Mod $p$ multiplicity one} \label{subsec:mod-p-multi-one}
We recall the mod $p$ multiplicity one result following Mazur, Wiles, Ribet, Edixhoven, and Faltings--Jordan.
The following form of the mod $p$ multiplicity one is due to Faltings--Jordan \cite[Theorem 2.1]{faltings-jordan}.
We write
$$ V_{k,\mathbb{F}_p}(Y_1(N)) = \mathrm{H}^1_{\mathrm{\acute{e}t}}(Y_1(N)_{\overline{\mathbb{Q}}}, \mathrm{Sym}^{k-2}_{\mathbb{F}_p}(\overline{\mathcal{H}}^1_p)) $$
where $\overline{\mathcal{H}}^1_p = \mathrm{R}^1\varpi_*\mathbb{F}_p$ is the \'{e}tale $\mathbb{F}_p$-sheaf on $Y_1(N)$ and $\varpi : E \to Y_1(N)$ is the universal elliptic curve as in $\S$\ref{subsubsec:construction}.
\begin{thm}[Faltings--Jordan] \label{thm:mod-p-multiplicity-one}
Suppose that  $\mathfrak{m}$ is a non-Eisenstein maximal ideal of $\mathbb{T}(N)$.
If $(N, p) = 1$ and $2 \leq k \leq p-1$, then
\begin{enumerate}
\item $V_{k, \mathbb{F}_p}(Y_1(N)) [\mathfrak{m}]$ is 2-dimensional over $\mathbb{F}$, and
\item $\mathbb{T}(N)_\mathfrak{m}$ is Gorenstein.
\end{enumerate}
\end{thm}
\begin{cor} \label{cor:mod-p-multiplicity-one}
Under the same assumptions in Theorem \ref{thm:mod-p-multiplicity-one}, we have
$$V_{k, \mathbb{Z}_p}(Y_1(N)) \otimes_{\mathbb{T}(N)} \mathbb{T}(N) /  \mathfrak{m}  \simeq  \overline{\rho} .$$
where $\overline{\rho}$ is the residual Galois representation corresponding to $\mathfrak{m}$.
\end{cor}
\begin{rem}
The conventions of the Galois representations in $\S$\ref{subsubsec:construction} and in \cite[$\S$2]{faltings-jordan} are dual to each other.
\end{rem}
\begin{proof}
Due to the argument given in \cite[$\S$2.1]{faltings-jordan} and the existence of integral perfect paring on $V_{k, \mathbb{Z}_p} (Y_1(N))_{\mathfrak{m}}$  as in \cite[Corollary 1.6]{diamond-flach-guo},
the subquotients of a Galois-stable filtration of
$V_{k, \mathbb{Z}_p}(Y_1(N))_{\mathfrak{m}} \otimes_{\mathbb{T}(N)_\mathfrak{m}} \mathbb{T}(N)_\mathfrak{m} /  \mathfrak{m}$
and
the subquotients of a Galois-stable filtration of
$\mathrm{Hom} ( V_{k, \mathbb{F}_p} (Y_1(N))_{\mathfrak{m}}[\mathfrak{m}] , \mathbb{F}(1-k))$
coincide up to order.
Because $\overline{\rho}$ is irreducible, we have
$$V_{k, \mathbb{Z}_p}(Y_1(N))_{\mathfrak{m}} \otimes_{\mathbb{T}(N)_\mathfrak{m}} \mathbb{T}(N)_\mathfrak{m} /  \mathfrak{m} \simeq \mathrm{Hom} ( V_{k, \mathbb{F}_p} (Y_1(N))_{\mathfrak{m}}[\mathfrak{m}] , \mathbb{F}(1-k)) .$$
Thus, we know
$$\mathrm{dim}_{\mathbb{F}} \left( V_{k, \mathbb{F}_p}(Y_1(N)) [\mathfrak{m}] \right) =
\mathrm{dim}_{\mathbb{F}} \left( V_{k, \mathbb{Z}_p}(Y_1(N)) \otimes_{\mathbb{T}(N)} \mathbb{T}(N) /  \mathfrak{m} \right) = 2 .$$
By the construction in $\S$\ref{subsubsec:construction}, it is isomorphic to $\overline{\rho}$.
\end{proof}

\subsection{Congruences between zeta elements: the same level} \label{subsec:same-level}
\subsubsection{}
Let $f = \sum_{n \geq1} a_n(f)q^n \in S_k(\Gamma_1(N_f))$ and $g = \sum_{n \geq1} a_n(g)q^n \in S_k(\Gamma_1(N_g))$ be newforms of weight $k \geq 2$. Assume that $f$ and $g$ are congruent modulo $\pi$ in the sense of the isomorphism between the residual representations $\overline{\rho}_f \simeq \overline{\rho}_g$.
Let $\Sigma_0$ be the set of finite places dividing $N_f$ and $N_g$.
We put
\begin{equation} \label{eqn:level_N_Sigma_0}
N^{\Sigma_0} = N(\overline{\rho}) \cdot \prod_{\ell} \ell \cdot \prod_{q} q^2
\end{equation}
where $\ell$ runs over the primes in $\Sigma_0$ such that
$\ell$ divides $N(\overline{\rho})$ exactly  once or $\ell \equiv 1 \pmod{p}$ and $\mathrm{ord}_\ell (N(\overline{\rho})) = \mathrm{ord}_\ell (N(\mathrm{det}(\overline{\rho})))$
and $q$ runs over the primes in $\Sigma_0$ not dividing $N(\overline{\rho})$.
Then
\[
\xymatrix{
f^{\Sigma_0} = \sum_{n \geq 1} a_n(f^{\Sigma_0})q^n \coloneqq  \sum_{(n, N^{\Sigma_0})=1} a_n(f)q^n  , & g^{\Sigma_0} = \sum_{n \geq 1} a_n(g^{\Sigma_0})q^n \coloneqq  \sum_{(n, N^{\Sigma_0})=1} a_n(g)q^n
}
\]
are eigenforms in $S_k(\Gamma_1(N^{\Sigma_0}))$. Then both $N_f$ and $N_g$ divides $N^{\Sigma_0}$ thanks to the following proposition.
\begin{prop}
For any newform $f \in S_k(\overline{\rho})$ of level $N_f$, $N_f$ divides $N^{\Sigma_0}$ of the form (\ref{eqn:level_N_Sigma_0}) where $\Sigma_0$ is the set of finite places dividing $N_f$.
\end{prop}
\begin{proof}
Since $\overline{\rho}$ is irreducible and $2 \leq k \leq p-1$, we can apply \cite[Theorem A]{diamond-taylor-non-optimal}.
Then $N_f$ must be of the following form
$$N_f = N(\overline{\rho}) \cdot \prod_\ell \ell^{\alpha(\ell)}.$$
For each prime $\ell$ with $\alpha(\ell) > 0$, one of the following statements hold:
\begin{enumerate}
\item $\ell \nmid N(\overline{\rho})$, $\ell \cdot \left( \mathrm{tr}\left( \overline{\rho}(\mathrm{Frob}_\ell) \right) \right)^2 = (1+\ell)^2 \cdot   \mathrm{det}\left( \overline{\rho}(\mathrm{Frob}_\ell) \right)$ in $\overline{\mathbb{F}}$, and $\alpha(\ell) = 1$.
\item $\ell \equiv -1 \pmod{p}$ and one of the following holds:
\begin{enumerate}
\item $\ell \nmid N(\overline{\rho})$, $\mathrm{tr}\left( \overline{\rho}(\mathrm{Frob}_\ell) \right) = 0$ in $\overline{\mathbb{F}}$,  and $\alpha(\ell) = 2$;
\item $\ell \Vert N(\overline{\rho})$, $\mathrm{det}(\overline{\rho})$ is unramified at $\ell$, and $\alpha(\ell)= 1$.
\end{enumerate}
\item $\ell \equiv 1 \pmod{p}$ and one of the following holds:
\begin{enumerate}
\item $\ell \nmid N(\overline{\rho})$, and $\alpha(\ell) = 2$;
\item $\ell^2 \nmid N(\overline{\rho})$ or the power of $\ell$ dividing $N(\overline{\rho})$ is the same as the power dividing the conductor of $\mathrm{det}(\overline{\rho})$, and $\alpha(\ell)= 1$.
\end{enumerate}
\end{enumerate}
Therefore, 
\begin{itemize}
\item if $\ell \nmid N(\overline{\rho})$, then $1 \leq \alpha(\ell) \leq 2$, and
\item if $\ell \Vert N(\overline{\rho})$ or $\ell \equiv 1 \pmod{p}$ and $\mathrm{ord}_\ell (N(\overline{\rho})) = \mathrm{ord}_\ell (N(\mathrm{det}(\overline{\rho})))$, then $\alpha(\ell) = 1$.
\end{itemize}
The conclusion follows.
\end{proof}
\subsubsection{}
We use the same notation $\mathbb{T}(N^{\Sigma_0})$ for its base change to $\mathcal{O}_\pi$.
Let $\mathfrak{m}^{\Sigma_0} \subseteq \mathbb{T}(N^{\Sigma_0})$ be the non-Eisenstein maximal ideal
generated by $\pi$, $T_q - a_q(f^{\Sigma_0})$ for all primes $q$ not dividing $N^{\Sigma_0}$, and $U_r$ for all primes $r$ dividing $N^{\Sigma_0}$.
Then $\mathfrak{m}^{\Sigma_0}$ contains both $\wp_{f^{\Sigma_0}}$ and $\wp_{g^{\Sigma_0}}$.

\begin{thm} \label{thm:mod-p-multi-one-vatsal}
If $\mathfrak{m}^{\Sigma_0}$ is non-Eisenstein,  $(N^{\Sigma_0}, p) = 1$, and $2 \leq k \leq p-1$, then
\begin{align*}
\mathrm{H}^1(\Gamma_1(N^{\Sigma_0}), \mathrm{Sym}^{k-2}(\mathcal{O}^2_\pi) )^{\pm}_{\mathfrak{m}^{\Sigma_0}}
& \simeq \mathrm{H}^1(Y_1(N^{\Sigma_0})(\mathbb{C}), \mathrm{Sym}^{k-2}(\mathcal{O}^2_\pi) )^{\pm}_{\mathfrak{m}^{\Sigma_0}} \\
& \simeq \mathrm{Hom}_{\mathcal{O}_\pi} \left( \mathbb{T}(N^{\Sigma_0})_{\mathfrak{m}^{\Sigma_0}} , \mathcal{O}_\pi \right) \\
& \simeq \mathbb{T}(N^{\Sigma_0})_{\mathfrak{m}^{\Sigma_0}} .
\end{align*}
\end{thm}
\begin{proof}
It follows from Theorem \ref{thm:mod-p-multiplicity-one}. See \cite[Theorem 1.13]{vatsal-cong}.
\end{proof}
By Theorem \ref{thm:mod-p-multi-one-vatsal}, the Galois module
$$T_{\mathfrak{m}^{\Sigma_0}} = \mathrm{H}^1(X_1(N^{\Sigma_0}), \mathrm{Sym}^{k-2}(\mathcal{O}^2_\pi))_{\mathfrak{m}^{\Sigma_0}} = V_{k,\mathbb{Z}_p}(Y_1(N^{\Sigma_0}))_{\mathfrak{m}^{\Sigma_0}} \otimes_{\mathbb{Z}_p} \mathcal{O}_\pi$$
(in the sense of $\S$4.1) is free of rank two over the localized Hecke algebra $\mathbb{T}(N^{\Sigma_0})_{\mathfrak{m}^{\Sigma_0}}$.
It is known that there is no distinction between the use of $X_1$ and $Y_1$ above after the localization at a non-Eisenstein maximal ideal.
We are able to write
$$T_{\mathfrak{m}^{\Sigma_0}} =  \mathbb{T}(N^{\Sigma_0})_{\mathfrak{m}^{\Sigma_0}} \delta^+_{\mathfrak{m}^{\Sigma_0}} + \mathbb{T}(N^{\Sigma_0})_{\mathfrak{m}^{\Sigma_0}} \delta^-_{\mathfrak{m}^{\Sigma_0}}$$
where $\delta^\pm_{\mathfrak{m}^{\Sigma_0}}$ is a $\mathbb{T}(N^{\Sigma_0})_{\mathfrak{m}^{\Sigma_0}}$-generator of $\mathrm{H}^1(X_1(N^{\Sigma_0}), \mathrm{Sym}^{k-2}(\mathcal{O}^2_\pi))^{\pm}_{\mathfrak{m}^{\Sigma_0}}$ such that $\iota (\delta^{\pm}_{\mathfrak{m}^{\Sigma_0}} ) = \pm \delta^{\pm}_{\mathfrak{m}^{\Sigma_0}}$ and $\iota$ is the involution induced by the complex conjugation.

For any eigenform $f^{\Sigma_0} \in S_k(\Gamma_1(N^{\Sigma_0}), \mathcal{O}_\pi )_{\mathfrak{m}^{\Sigma_0}}$, we also have the following identifications and the natural quotient map
\begin{align} \label{eqn:quotient-integral-galois-modules}
\begin{split}
\mathrm{H}^1(\Gamma_1(N^{\Sigma_0}), \mathrm{Sym}^{k-2}(\mathcal{O}^2_\pi) )^{\pm}_{\mathfrak{m}^{\Sigma_0}}
& \simeq \mathrm{H}^1(Y_1(N^{\Sigma_0})(\mathbb{C}), \mathrm{Sym}^{k-2}(\mathcal{O}^2_\pi) )^{\pm}_{\mathfrak{m}^{\Sigma_0}} \\
& \simeq V_{k,\mathbb{Z}_p}(Y_1(N^{\Sigma_0}))^{\pm}_{\mathfrak{m}^{\Sigma_0}} \\
& \to V_{k,\mathbb{Z}_p}(Y_1(N^{\Sigma_0}))^{\pm}_{\mathfrak{m}^{\Sigma_0}} \otimes \mathbb{T}(N^{\Sigma_0})_{\mathfrak{m}^{\Sigma_0}} / \wp_{f^{\Sigma_0}} \\
& = T^{\pm}_{f^{\Sigma_0}} .
\end{split}
\end{align}
By Chebotarev density theorem, we have $T_f \simeq T_{f^{\Sigma_0}}$ as Galois modules.

Under the assumptions in Theorem \ref{thm:mod-p-multi-one-vatsal}, we have an integral Eichler--Shimura isomorphism
$$S_k(\Gamma_1(N^{\Sigma_0}), \mathcal{O}_\pi )_{\mathfrak{m}^{\Sigma_0}} \simeq
\mathrm{H}^1(\Gamma_1(N^{\Sigma_0}), \mathrm{Sym}^{k-2}(\mathcal{O}^2_\pi) )^{\pm}_{\mathfrak{m}^{\Sigma_0}}$$
under which \emph{normalized} eigenforms $f^{\Sigma_0}$ and $g^{\Sigma_0}$ maps to the canonical cohomology classes $\delta^{\pm}_{f^{\Sigma_0}}$ and $\delta^{\pm}_{g^{\Sigma_0}}$
in group cohomology $\mathrm{H}^1(\Gamma_1(N^{\Sigma_0}), \mathrm{Sym}^{k-2}(\mathcal{O}^2_\pi) )^{\pm}_{\mathfrak{m}^{\Sigma_0}}$, respectively, and these cohomology classes yield the canonical periods. See  \cite[$\S$1.3]{vatsal-cong}  and \cite[Definition 3.4]{vatsal-integralperiods-2013}.
This assignment works with eigenforms which are not necessarily newforms as explained in \cite[$\S$3.2]{vatsal-integralperiods-2013}. We denote by the same notation the image of $\delta^{\pm}_{f^{\Sigma_0}}$ in $T^{\pm}_{f^{\Sigma_0}}$.
The same rule applies to $g^{\Sigma_0}$.

Since $f^{\Sigma_0}$ is a normalized eigenform and the Eichler--Shimura isomorphism is integral, the corresponding cohomology class $\delta^{\pm}_{f^{\Sigma_0}}$ does not vanish modulo $\pi$.
Therefore, $\delta^{\pm}_{f^{\Sigma_0}}$ generates $T^{\pm}_{f^{\Sigma_0}}$ so that $\delta^{\pm}_{f^{\Sigma_0}}$ can play the role of $\gamma^{\pm}_{f^{\Sigma_0}}$ in (\ref{eqn:zeta-elements}).
As before, the same rule applies to $g^{\Sigma_0}$.

Since $f^{\Sigma_0} \equiv g^{\Sigma_0} \pmod{\pi}$, we have congruences between $\delta^{\pm}_{f^{\Sigma_0}} \in T^{\pm}_{f^{\Sigma_0}}$ and $\delta^{\pm}_{g^{\Sigma_0}} \in T^{\pm}_{g^{\Sigma_0}}$ in the residual representation as in \cite[(4), Page 402]{vatsal-cong} and \cite[Proposition 3.6]{vatsal-integralperiods-2013}, i.e.
\begin{equation} \label{eqn:congruences-same-levels}
\delta^{\pm}_{f^{\Sigma_0}} \equiv u^{\pm} \cdot \delta^{\pm}_{g^{\Sigma_0}} \pmod{\pi}
\end{equation}
where $u^{\pm} \in \mathcal{O}^\times_\pi$ in the residual representation following the isomorphisms
\begin{align*}
\overline{\rho}^{\pm} & \simeq  V_{k, \mathbb{Z}_p}(Y_1(N^{\Sigma_0}))^{\pm}_{\mathfrak{m}^{\Sigma_0}} \otimes_{\mathbb{T}(N^{\Sigma_0})_{\mathfrak{m}^{\Sigma_0}}} \mathbb{T}(N^{\Sigma_0})_{\mathfrak{m}^{\Sigma_0}} /  \wp_{f^{\Sigma_0}} \otimes \mathcal{O}_\pi/  \pi  \\
&  \simeq   V_{k, \mathbb{Z}_p}(Y_1(N^{\Sigma_0}))^{\pm}_{\mathfrak{m}^{\Sigma_0}} \otimes_{\mathbb{T}(N^{\Sigma_0})_{\mathfrak{m}^{\Sigma_0}}} \mathbb{T}(N^{\Sigma_0})_{\mathfrak{m}^{\Sigma_0}} / \wp_{g^{\Sigma_0}} \otimes \mathcal{O}_\pi / \pi
\end{align*}
which follow from Corollary \ref{cor:mod-p-multiplicity-one}.

\subsubsection{}
We briefly review how to construct zeta elements associated to $\delta_{f^{\Sigma_0}} = \delta^+_{f^{\Sigma_0}} + \delta^-_{f^{\Sigma_0}} \in T_{f^{\Sigma_0}}$.
Following the explicit construction in \cite[$\S$5.5]{kato-euler-systems}, we have the cocycle
$$\delta_{1,N^{\Sigma_0}}(k ,j, \alpha)^\pm \in V_{k,\mathbb{Q}}(Y_1(N^{\Sigma_0}))^{\pm} \coloneqq  \mathrm{H}^1(Y_1(N^{\Sigma_0})(\mathbb{C}), \mathrm{Sym}^{k-2}(\mathbb{Q}^2) )^{\pm} .$$
Write $\delta(f^{\Sigma_0},j, \alpha)^\pm \in V^{\pm}_{f^{\Sigma_0}}$ to be the image of
$\delta_{1,N^{\Sigma_0}}(k ,j, \alpha)^\pm$ in $V^{\pm}_{f^{\Sigma_0}}$ as in \cite[$\S$6.3]{kato-euler-systems}.
Following \cite[$\S$13.9]{kato-euler-systems}, we are able to write
\begin{align} \label{eqn:gamma-delta}
\begin{split}
\delta_{f^{\Sigma_0}} & = \delta^+_{f^{\Sigma_0}} +
 \delta^-_{f^{\Sigma_0}}  \\
 & = b_1(f^{\Sigma_0})  \cdot \delta(f^{\Sigma_0},j_1, \alpha_1)^+ +
 b_2(f^{\Sigma_0})  \cdot \delta(f^{\Sigma_0},j_2, \alpha_2)^- , \\
\delta_{g^{\Sigma_0}} & = \delta^+_{g^{\Sigma_0}} + \delta^-_{g^{\Sigma_0}} \\
 & = b_1(g^{\Sigma_0})  \cdot \delta(g^{\Sigma_0},j'_1, \alpha'_1)^+ +
 b_2(g^{\Sigma_0})  \cdot \delta(g^{\Sigma_0},j'_2, \alpha'_2)^-
\end{split}
\end{align}
in $T_{f^{\Sigma_0}}$ and $T_{g^{\Sigma_0}}$, respectively, where
$b_1(f^{\Sigma_0}) ,b_2(f^{\Sigma_0}) , b_1(g^{\Sigma_0})$, and $b_2(g^{\Sigma_0})$ lie in $F_\pi$.
Here, $\alpha_i , \alpha'_i \in \mathrm{SL}_2(\mathbb{Z})$ and $j_i, j'_i$ are integers such that $1 \leq j_i, j'_i \leq k-1$ ($i=1, 2$).
See also \cite[Theorem 13.6]{kato-euler-systems} for the integrality of $\delta(f^{\Sigma_0},j_i, \alpha_i)^\pm$ and $\delta(g^{\Sigma_0},j'_i, \alpha'_i)^\pm$ for $i = 1, 2$.

As in \cite[$\S$13.9]{kato-euler-systems}, the zeta element
 $\mathbf{z}_{\mathrm{Kato}}(f^{\Sigma_0}, k-r)$ associated to
 $\delta_{f^{\Sigma_0}} = b_1(f^{\Sigma_0})  \cdot \delta(f^{\Sigma_0},j_1, \alpha_1)^+ +
 b_2(f^{\Sigma_0})  \cdot \delta(f^{\Sigma_0},j_2, \alpha_2)^- $
 is defined by
\begin{align} \label{eqn:zeta-morphism-single-explicit}
\begin{split}
&  \mathbf{z}_{\mathrm{Kato}}(f^{\Sigma_0}, k-r) \\
= & \left( \left\lbrace \mu(c,d,j_1)^{-1} \cdot b_1(f^{\Sigma_0}) \cdot \left( {}_{c,d}z^{(p)}_{p^n} (f^{\Sigma_0} , k, j_1, \alpha_1, pN^{\Sigma_0}) \right)_{n \geq 1}   \right\rbrace^{-} \right. \\
& +
\left. \left\lbrace  \mu(c,d,j_2)^{-1} \cdot  b_2(f^{\Sigma_0}) \cdot \left( {}_{c,d}z^{(p)}_{p^n} (f^{\Sigma_0} , k, j_2, \alpha_2, pN^{\Sigma_0}) \right)_{n \geq 1} \right\rbrace^{+} \right) \otimes \left( \zeta_{p^n} \right)^{\otimes k-r}_{n \geq 1}
\end{split}
\end{align}
where
$\mu(c,d,j) = \left( c^2 - c^{k+1-j} \cdot \sigma_c \right) \cdot \left( d^2 - d^{j+1} \cdot \sigma_d \right)  \in \Lambda$ and
$\left( {}_{c,d}z^{(p)}_{p^n} (f^{\Sigma_0} , k, j, \alpha, pN^{\Sigma_0}) \right)_{n \geq 1}$ is the integral zeta element appeared in \cite[(8.1.3)]{kato-euler-systems} with $c, d, k, j \in \mathbb{Z}$ and $\alpha \in \mathrm{SL}_2(\mathbb{Z})$ satisfying certain conditions in \cite[(5.2.1)--(5.2.3) with the $\xi \in \mathrm{SL}_2(\mathbb{Z})$ case in $\S$5.2]{kato-euler-systems}.
Of course, $\mathbf{z}_{\mathrm{Kato}}(g^{\Sigma_0}, k-r)$ is defined in the exactly same way.
Then
$$\mathbf{z}_{\mathrm{Kato}}(f^{\Sigma_0}, k-r) \in \mathbf{H}^1(j_*  T_{f^{\Sigma_0}}(k-r)) \otimes Q(\Lambda)$$ \emph{a priori} where $Q(\Lambda)$ is the total quotient ring of $\Lambda$.
In particular, the assignment $\delta_{f^{\Sigma_0}} \mapsto \mathbf{z}_{\mathrm{Kato}}(f^{\Sigma_0}, k-r)$ in (\ref{eqn:zeta-morphism-single-explicit}) is independent of all the choices and it is canonical \cite[Theorem 12.5.(1)]{kato-euler-systems}. 
All the choices are made only for the construction of $\delta_{f^{\Sigma_0}}$ in (\ref{eqn:gamma-delta}).
Thanks to \cite[Theorem 12.5.(1)]{kato-euler-systems} again, this assignment extends to a homomorphism (zeta morphism) with the smaller target
\begin{equation} \label{eqn:zeta-morphism-single-modular-form-rational}
\mathbf{z}_{\mathrm{Kato}}  \otimes (\zeta_{p^n})^{\otimes (k-r)}_n : V_{f^{\Sigma_0}} \to \mathbf{H}^1(j_*  V_{f^{\Sigma_0}}(k-r)) \subseteq \mathbf{H}^1(j_*  T_{f^{\Sigma_0}}(k-r)) \otimes Q(\Lambda)
\end{equation}
so that $\mathbf{z}_{\mathrm{Kato}}(f^{\Sigma_0}, k-r) \in \mathbf{H}^1(j_*  V_{f^{\Sigma_0}}(k-r))$.
Another explanation of the zeta morphism can be found in \cite[$\S$3.3; in particular, (30)]{nakamura-kato-deformation}. Under Assumption \ref{assu:image},
\cite[Theorem 12.5.(4)]{kato-euler-systems} implies that we have the integral zeta morphism
\begin{equation} \label{eqn:zeta-morphism-single-modular-form-integral}
\mathbf{z}_{\mathrm{Kato}}  \otimes (\zeta_{p^n})^{\otimes (k-r)}_n : T_{f^{\Sigma_0}} \to \mathbf{H}^1(j_*  T_{f^{\Sigma_0}}(k-r))
\end{equation}
so that $\mathbf{z}_{\mathrm{Kato}}(f^{\Sigma_0}, k-r) \in \mathbf{H}^1(j_*  T_{f^{\Sigma_0}}(k-r))$.

Combining with the projection to the $\omega^i$-isotypic component,
we have the zeta morphism
$$T_{f^{\Sigma_0}} \to \mathbf{H}^1(j_*  T_{f^{\Sigma_0}}(k-r)) \to \mathbb{H}^1(j_*  T_{f^{\Sigma_0},i}(k-r))$$
The mod $\pi$ reduction of this zeta morphism yields the following commutative diagram
\begin{equation} \label{eqn:commutative-diagram-reduction}
\begin{split}
\xymatrix@C=1em{
T_{f^{\Sigma_0}} \ar[r] \ar[dd] & \mathbf{H}^1(j_*  T_{f^{\Sigma_0}}(k-r)) \ar[r] \ar[d] & \mathbb{H}^1(j_*  T_{f^{\Sigma_0},i}(k-r)) \ar[d]  \\
& \mathrm{H}^1_{\mathrm{Iw}}(\mathbb{Q}_\Sigma/\mathbb{Q}(\zeta_{p^\infty}), T_{f^{\Sigma_0}}(k-r)) \ar[r] \ar[d] & \mathrm{H}^1_{\mathrm{Iw}}(\mathbb{Q}_\Sigma/\mathbb{Q}_{\infty}, T_{f^{\Sigma_0},i}(k-r))  \ar[d] \\
\overline{\rho} \simeq T_{f^{\Sigma_0}} / \pi \ar[r] \ar[dr] & \mathrm{H}^1_{\mathrm{Iw}}(\mathbb{Q}_\Sigma/\mathbb{Q}(\zeta_{p^\infty}), T_{f^{\Sigma_0}}(k-r)) / \pi \ar[r] \ar@{_{(}->}[d]  & \mathrm{H}^1_{\mathrm{Iw}}(\mathbb{Q}_\Sigma/\mathbb{Q}_{\infty}, T_{f^{\Sigma_0},i}(k-r)) / \pi \ar@{_{(}->}[d] \\
& \mathrm{H}^1_{\mathrm{Iw}}(\mathbb{Q}_\Sigma/\mathbb{Q}(\zeta_{p^\infty}), \overline{\rho}_{f^{\Sigma_0}}(k-r)) \ar[r] & \mathrm{H}^1_{\mathrm{Iw}}(\mathbb{Q}_\Sigma/\mathbb{Q}_{\infty}, \overline{\rho}_{f^{\Sigma_0},i}(k-r))
}
\end{split}
\end{equation}
where $\Sigma = \Sigma_0 \cup \lbrace p, \infty \rbrace$.
\begin{rem} \label{rem:zeta-elements-general-abelian}
For an integer $m \geq 1$, we can also define the zeta element over $\mathbb{Q}(\zeta_{mp^\infty})$ associated to $f^{\Sigma_0}$
$$ \mathbf{z}_{\mathrm{Kato}, \mathbb{Q}(\zeta_m)}(f^{\Sigma_0}, k-r) \in \mathrm{H}^1_{\mathrm{Iw}}(\mathbb{Q}_\Sigma/\mathbb{Q}(\zeta_{mp^\infty}), T_{f^{\Sigma_0}}(k-r))$$
by using the same formula as in (\ref{eqn:zeta-morphism-single-explicit}) but with
\[
\xymatrix{
\left( {}_{c,d}z^{(p)}_{mp^n} (f^{\Sigma_0} , k, j_1, \alpha_1, pN^{\Sigma_0}) \right)_{n \geq 1} , & \left( {}_{c,d}z^{(p)}_{mp^n} (f^{\Sigma_0} , k, j_2, \alpha_2, pN^{\Sigma_0}) \right)_{n \geq 1} .
}
\]
This assignment also yields a similar zeta morphism from $T_{f^{\Sigma_0}}$ to $\mathrm{H}^1_{\mathrm{Iw}}(\mathbb{Q}_\Sigma/\mathbb{Q}(\zeta_{mp^\infty}), T_{f^{\Sigma_0}}(k-r))$.
It is not difficult to check the integrality of this zeta element (and morphism) by using the $\Lambda$-freeness of
$\mathrm{H}^1_{\mathrm{Iw}}(\mathbb{Q}_\Sigma/\mathbb{Q}(\zeta_{mp^\infty}), T_{f^{\Sigma_0}}(k-r))$ under the large image assumption (Assumption \ref{assu:image}). See \cite[$\S$13.14]{kato-euler-systems} and \cite{kim-kato} for example.
\end{rem}
\subsubsection{}
For notational convenience, we use $f$ and $N$ instead of $f^{\Sigma_0}$ and $N^{\Sigma_0}$.
We first review the construction of families of Kato's zeta elements (zeta morphisms).
Let $T_{\mathfrak{m}} = \mathrm{H}^1(X_1(N), \mathrm{Sym}^{k-2}(\mathcal{O}^2_\pi))_{\mathfrak{m}}$
 and $V_\mathfrak{m} = T_{\mathfrak{m}} \otimes \mathbb{Q}_p$.
Let $\delta_{1, N}(k, r, \xi)_{\mathfrak{m}}  \in \mathrm{H}^1(X_1(N), \mathrm{Sym}^{k-2}(\mathbb{Q}^2))_{\mathfrak{m}}$ be the $\mathfrak{m}$-component of $\delta_{1, N}(k, r, \xi)$ where $\xi \in \mathrm{SL}_2(\mathbb{Z})$ \cite[$\S$5.5]{kato-euler-systems}.
By using the idea of \cite[Lemma 13.10.(2)]{kato-euler-systems} (cf. \cite[$\S$3]{fukaya-kato-sharificonj}, \cite[$\S$3.8]{nakamura-kato-deformation}), consider the assignment
\begin{align} \label{eqn:zeta-morphism-families-explicit}
\begin{split}
\delta_{1, N}(k, r, \xi)_{\mathfrak{m}} \mapsto & \ ( \left( \left( c^2 - c^{k+1 -r} \cdot \sigma_c \right) \cdot  \left( d^2 - d^{ r+1} \cdot \sigma_d \right)  \right)^{-1}  \\
& \cdot {}_{c,d} z^{(p)}_{1,N, p^n} (k, k, r , \xi, pN)_{\mathfrak{m}} \otimes (\zeta_{p^n})^{\otimes (k-r)} )_n .
\end{split} 
\end{align}
Note that $\left( c^2 - c^{k+1 -r} \cdot \sigma_c \right) \cdot  \left( d^2 - d^{ r+1} \cdot \sigma_d \right)$ is independent of any choice coming from the Hecke modules above.
Since
$\delta^+_{\mathfrak{m}}$ and $\delta^-_{\mathfrak{m}}$ generate $V_{\mathfrak{m}}$ over $\mathbb{T}(N)_{\mathfrak{m}} \otimes \mathbb{Q}_p$, $\delta_{1, N}(k, r, \xi)^{\pm}_{\mathfrak{m}}$ is a $\mathbb{T}(N)_{\mathfrak{m}} \otimes \mathbb{Q}_p$-multiple of $\delta^{\pm}_{\mathfrak{m}}$.
In particular, all the relations among $\left\lbrace \delta_{1, N}(k, r, \xi)^{\pm}_{\mathfrak{m}} \right\rbrace_{(r, \xi)}$ in $V_{\mathfrak{m}}$ reduce to the $\mathbb{T}(N)_{\mathfrak{m}} \otimes \mathbb{Q}_p$-linear relations with $\delta^{\pm}_{\mathfrak{m}}$.
Thus, the $\mathbb{T}(N)_{\mathfrak{m}}$-linear extension of (\ref{eqn:zeta-morphism-families-explicit}) yields the zeta morphism parametrizing the families of Kato's zeta elements
\begin{equation} \label{eqn:zeta-morphism-families}
\mathbf{z}_{\mathrm{Kato}}  \otimes (\zeta_{p^n})^{\otimes (k-r)}_n: V_{\mathfrak{m}}  \to \mathbf{H}^1(j_*V_{\mathfrak{m}} (k-r)) \otimes_{\Lambda} Q(\Lambda \otimes \mathbb{T}(N)_{\mathfrak{m}} \otimes \mathbb{Q}_p) 
\end{equation}
and it is also independent of any choice. The Hecke-equivariant property of (\ref{eqn:zeta-morphism-families}) can also be found in \cite[Corollary 3.10]{nakamura-kato-deformation}.
We use the same notation for the zeta morphism for a single modular form.
Comparing with \cite[Lemma 13.10.(2)]{kato-euler-systems} more precisely, the Euler factors $\prod_{\ell \vert N} \left( 1 - T'(\ell) \cdot \ell^{-k} \cdot \sigma^{-1}_{\ell} \right)$ is missing in our formulation.
Here, ${}_{c,d} z^{(p)}_{1,N, p^n} (k, k, r , \xi, pN)_{\mathfrak{m}}$ is the $\mathfrak{m}$-component of ${}_{c,d} z^{(p)}_{1,N, p^n} (k, k, r , \xi, pN)$ \cite[$\S$8.1.2]{kato-euler-systems}.

We first show that the image of the map (\ref{eqn:zeta-morphism-families}) actually lies in
$\mathbf{H}^1(j_*V_{\mathfrak{m}} (k-r))$.
Recall that we have the decomposition of the rational Hecke algebra into the product of Hecke fields
$\mathbb{T}(N)_{\mathfrak{m}} \otimes \mathbb{Q}_p \simeq \prod_{f'} F(f')_{\pi}$
where $F(f')_{\pi} = \mathbb{Q}_p (\iota_p(a_n(f')):n)$ and $f'$ runs over the set of congruent eigenforms of fixed level $N$. This decomposition follows from the reduced property of $\mathbb{T}(N)_{\mathfrak{m}}$ as mentioned in $\S$\ref{subsec:hecke-algebras}.
Since $V_{\mathfrak{m}}$ is a free module over $\mathbb{T}(N)_{\mathfrak{m}} \otimes \mathbb{Q}_p$ of rank two,
we also have decomposition
$V_{\mathfrak{m}} \simeq \prod_{f'} V_{f'}$
as Galois--Hecke modules.
Since the zeta element of a single modular form is defined by the image of the zeta element of the space of modular forms \cite[$\S$8.11]{kato-euler-systems},
we have the commutative diagram
\[
\xymatrix{
V_{\mathfrak{m}} \ar[rrr]^-{\mathbf{z}_{\mathrm{Kato}}  \otimes (\zeta_{p^n})^{\otimes (k-r)}_n}_-{(\ref{eqn:zeta-morphism-families})} \ar[d]^-{\otimes \mathbb{T}(N)_{\mathfrak{m}} / \wp_{f'}} & & & \mathbf{H}^1(j_*V_{\mathfrak{m}} (k-r)) \otimes_{\Lambda} Q(\Lambda \otimes \mathbb{T}(N)_{\mathfrak{m}} \otimes \mathbb{Q}_p) \ar[d]^-{(*)} \\
V_{f'} \ar[rrr]^-{\mathbf{z}_{\mathrm{Kato}}  \otimes (\zeta_{p^n})^{\otimes (k-r)}_n}_-{(\ref{eqn:zeta-morphism-single-modular-form-rational})} & & & \mathbf{H}^1(j_* V_{f'} (k-r))  \subseteq \mathbf{H}^1(j_*V_{\mathfrak{m}} (k-r)) \otimes_{\Lambda} Q(\Lambda)
}
\]
for such a congruent eigenform $f'$ 
where the map $(*)$ is also induced from the map $\otimes \mathbb{T}(N)_{\mathfrak{m}} / \wp_{f'}$.

By using this compatibility with decomposition $\mathbb{T}(N)_{\mathfrak{m}} \otimes \mathbb{Q}_p \simeq \prod_{f'} F(f')_{\pi}$, the zeta morphism (\ref{eqn:zeta-morphism-families}) also decomposes into the product of the zeta morphisms of congruent eigenforms
\begin{equation} \label{eqn:zeta-morphism-families-reduced}
\mathbf{z}_{\mathrm{Kato}}  \otimes (\zeta_{p^n})^{\otimes (k-r)}_n: V_{\mathfrak{m}} \simeq \prod_{f'} V_{f'} \to  \prod_{f'} \mathbf{H}^1(j_* V_{f'} (k-r))  \simeq \mathbf{H}^1(j_*V_{\mathfrak{m}} (k-r)),
\end{equation}
so the target of (\ref{eqn:zeta-morphism-families}) becomes $\mathbf{H}^1(j_*V_{\mathfrak{m}} (k-r))$ thanks to (\ref{eqn:zeta-morphism-single-modular-form-rational}).

\subsubsection{}
We continue the simplified notation.
We integrally refine (\ref{eqn:zeta-morphism-families-reduced}) closely following the strategy of \cite[the proof of Proposition 3.15 (pp. 233)]{nakamura-kato-deformation}. This argument generalizes \cite[\S13.14]{kato-euler-systems}.
More precisely, we reduce the integrality of the families of the zeta morphism
\begin{equation} \label{eqn:zeta-morphism-families-reduced-integral}
\mathbf{z}_{\mathrm{Kato}}  \otimes (\zeta_{p^n})^{\otimes (k-r)}_n: T_{\mathfrak{m}} \to \mathbf{H}^1(j_*T_{\mathfrak{m}} (k-r))
\end{equation}
to the rational one (\ref{eqn:zeta-morphism-families-reduced}).

Recall that $\Lambda_0 = \mathcal{O}_\pi \llbracket \Gamma \rrbracket \simeq \mathcal{O}_\pi \llbracket \mathrm{Gal}(\mathbb{Q}(\zeta_{p^\infty})/\mathbb{Q}(\zeta_{p})) \rrbracket$.
For any $c \in \mathbb{Z}$ with $c \equiv 1 \pmod{p}$, we can regard $\sigma_c \in \Lambda$ as an element in $\Lambda_0$.
Due to the absolute irreducibility of $\overline{\rho}$, we have
$$\mathrm{H}^0(\mathbb{Q}, T_{\mathfrak{m}} (k-r) \widehat{\otimes}_{\mathbb{Z}_p} \Lambda \otimes_{\Lambda_0} \mathbb{F} ) = 0 $$
where $\widehat{\otimes}_{\mathbb{Z}_p}$ means the completed tensor product.
This vanishing implies that
$\mathbf{H}^1(j_*T_{\mathfrak{m}} (k-r))$ is finite free over $\Lambda_0$ by the same argument as in
\cite[the proof of Theorem 12.4.(3)]{kato-euler-systems}.
From the explicit formula of the zeta morphism (\ref{eqn:zeta-morphism-families-explicit}),
we have inclusion
$$\mathfrak{d} \cdot \left( \mathbf{z}_{\mathrm{Kato}}  \otimes (\zeta_{p^n})^{\otimes (k-r)}_n \right) \left( T_{\mathfrak{m}} \right)
\subseteq \mathbf{H}^1(j_*T_{\mathfrak{m}} (k-r)) $$
for
$$\mathfrak{d} = \prod^{k-1}_{r=1} \left( c^2 - c^{k+1 -r} \cdot \sigma_c \right) \cdot  \left( d^2 - d^{ r+1} \cdot \sigma_d \right) \in \Lambda$$
where $c$ and $d$ are any integers such that
$c \equiv 1 \pmod{Np}$, $d \equiv 1 \pmod{Np}$, and $(cd, 6)=1$ (cf. \cite[(8.1.2)]{kato-euler-systems}).
Then $\mathfrak{d}$ is actually an element of $\Lambda_0$ and is not divisible by $p$ in $\Lambda_0$.
Therefore, the integrality (\ref{eqn:zeta-morphism-families-reduced-integral}) follows from the $p$-inverted one (\ref{eqn:zeta-morphism-families-reduced}).

\subsubsection{}
Since the zeta element for a single modular form is defined by the image of the zeta element for the space of modular forms \cite[$\S$8.11]{kato-euler-systems} under the quotient map by $\wp_{f'}$, we have the compatibility of zeta morphisms
\[
\xymatrix{
T_{\mathfrak{m}}  \ar[rr]^-{(\ref{eqn:zeta-morphism-families-reduced-integral})}
\ar[d]^-{ \prod_{f'} (-) \otimes \mathbb{T}(N)_{\mathfrak{m}} / \wp_{f'} }
& &  \mathbf{H}^1(j_*T_{\mathfrak{m}} (k-r))  \ar[d]^-{ \prod_{f'} (-) \otimes \mathbb{T}(N)_{\mathfrak{m}} / \wp_{f'} } \ar[r]^-{(**)} &  \prod_{i} \mathbb{H}^1(j_*T_{\mathfrak{m}, i} (k-r)) \ar[d]^-{ \prod_{f'} (-) \otimes \mathbb{T}(N)_{\mathfrak{m}} / \wp_{f'} }
 \\
 \prod_{f'} T_{f'}   \ar[rr]^-{\prod_{f'}(\ref{eqn:zeta-morphism-single-modular-form-integral})_{f'}} \ar[d]^-{ \prod_{f'} (-) \otimes \mathcal{O}_{\pi}/\pi } &  &  \prod_{f'} \mathbf{H}^1(j_*T_{f'} (k-r))  \ar[d]^-{ \prod_{f'} (-) \otimes \mathcal{O}_{\pi}/\pi } \ar[r]^-{(**)} &
\prod_{f'} \prod_{i} \mathbb{H}^1(j_*T_{f',i} (k-r)) \ar[d]^-{ \prod_{f'} (-) \otimes \mathcal{O}_{\pi}/\pi }
\\
  \prod_{f'} \overline{\rho}   \ar[rr]^-{\prod_{f'}(*)_{f'}}  &  &  \prod_{f'} \mathrm{H}^1_{\mathrm{Iw}}(\mathbb{Q}_\Sigma/\mathbb{Q}(\zeta_{p^\infty}), \overline{\rho}(k-r)) \ar[r]^-{(**)} &
\prod_{f'} \prod_{i}  \mathrm{H}^1_{\mathrm{Iw}}(\mathbb{Q}_\Sigma/\mathbb{Q}_{\infty}, \overline{\rho}_i(k-r))
}
\]
where $(-) \otimes \mathbb{T}(N)_{\mathfrak{m}}/ \wp_{f'}$ is viewed as the projection to $f'$,
$(-) \otimes \mathcal{O}_{\pi}/\pi$ is the mod $\pi$ reduction of each projection,
the $(*)_{f'}$ is the bottom map in the above diagram (\ref{eqn:commutative-diagram-reduction}), and
the $(**)$'s are the maps induced from the $\omega^i$-isotypic decomposition.
Since the $\xi = a(A)$ case can be done similarly, we omit it.
For a $\mathbb{T}(N)_{\mathfrak{m}}$-module $M$, we have
\begin{equation} \label{eqn:mod-m-reduction-hecke-modules}
 M \otimes_{\mathbb{T}(N)_{\mathfrak{m}}} \mathbb{T}(N)_{\mathfrak{m}}/ \wp_{f'} \otimes_{\mathcal{O}_{\pi}} \mathcal{O}_{\pi}/\pi  = M \otimes_{\mathbb{T}(N)_{\mathfrak{m}}} \mathbb{T}(N)_{\mathfrak{m}}/ \mathfrak{m} .
\end{equation}
In particular,  the map  $(*)_{f'}$ is as the same map as $(\ref{eqn:zeta-morphism-families-reduced-integral}) \otimes_{\mathbb{T}(N)_{\mathfrak{m}}} \mathbb{T}(N)_{\mathfrak{m}}/ \mathfrak{m}$ for every $f'$, so $(*)_{f'}$ is independent of $f'$.
Under the above diagram, the image of $\delta^{\pm}_{\mathfrak{m}}$ is computed as follows:
\[
\xymatrix{
\delta^{\pm}_{\mathfrak{m}} \ar@{|->}[rr]^-{(\ref{eqn:zeta-morphism-families-reduced-integral})}
\ar@{|->}[d]^-{ \prod_{f'} (-) \otimes \mathbb{T}(N)_{\mathfrak{m}} / \wp_{f'} }
& &  \left( \mathbf{z}_{\mathrm{Kato}}  \otimes (\zeta_{p^n})^{\otimes (k-r)}_n \right) \left( \delta^{\pm}_{\mathfrak{m}} \right)  \ar@{|->}[d]^-{ \prod_{f'} (-) \otimes \mathbb{T}(N)_{\mathfrak{m}} / \wp_{f'} } \ar@{|->}[r]^-{(**)} &  \left( \left( \mathbf{z}_{\mathrm{Kato}}  \otimes (\zeta_{p^n})^{\otimes (k-r)}_n \right) \left( \delta^{\pm}_{\mathfrak{m}} \right)_i \right)_i \ar@{|->}[d]^-{ \prod_{f'} (-) \otimes \mathbb{T}(N)_{\mathfrak{m}} / \wp_{f'} }
 \\
 \left( \delta^{\pm}_{f'} \right)_{f'}   \ar@{|->}[rr]^-{\prod_{f'}(\ref{eqn:zeta-morphism-single-modular-form-integral})_{f'}} \ar@{|->}[d]^-{ \prod_{f'} (-) \otimes \mathcal{O}_{\pi}/\pi } &  &  \left( \mathbf{z}_{\mathrm{Kato}}(f', k-r)^{\pm} \right)_{f'}   \ar@{|->}[d]^-{ \prod_{f'} (-) \otimes \mathcal{O}_{\pi}/\pi } \ar@{|->}[r]^-{(**)} &
 \left( \mathbf{z}_{\mathrm{Kato}}(f', i, k-r)^{\pm} \right)_{f',i}  \ar@{|->}[d]^-{ \prod_{f'} (-) \otimes \mathcal{O}_{\pi}/\pi }
\\
  \prod_{f'} \overline{\delta^{\pm}_{f'}}  \ar@{|->}[rr]^-{\prod_{f'}(*)_{f'}}  &  &  \left( \overline{\mathbf{z}_{\mathrm{Kato}}(f', k-r)^{\pm}} \right)_{f'}  \ar@{|->}[r]^-{(**)} &
 \left( \overline{\mathbf{z}_{\mathrm{Kato}}(f', i, k-r)^{\pm}} \right)_{f',i}  .
}
\]
Due to (\ref{eqn:mod-m-reduction-hecke-modules}),
we have
$\overline{\delta^{\pm}_{f'}} = \delta^{\pm}_{\mathfrak{m}} \pmod{\mathfrak{m}}$
in $\overline{\rho} =  T_{f'}/\pi = T_{\mathfrak{m}} /\mathfrak{m}$,
so $\overline{\delta^{\pm}_{f'}}$ is also independent of $f'$.
Therefore, the image of $\overline{\delta^{\pm}_{f'}}$ under $(*)_{f'}$ is also independent of $f'$.
To sum up, we obtain the following statement.

\begin{prop} \label{prop:independent_of_f}
Suppose that $\mathfrak{m}^{\Sigma_0}$ is non-Eisenstein.
The image of $\delta^{\pm}_{\mathfrak{m}^{\Sigma_0}} \pmod{\mathfrak{m}^{\Sigma_0}}$
under the composition of the maps
$$T_{\mathfrak{m}^{\Sigma_0}}/\mathfrak{m}^{\Sigma_0}  \simeq  T_{f^{\Sigma_0}} / \pi \simeq \overline{\rho} \to \mathrm{H}^1_{\mathrm{Iw}}(\mathbb{Q}_\Sigma/\mathbb{Q}(\zeta_{p^\infty}), \overline{\rho}(k-r)) \to \mathrm{H}^1_{\mathrm{Iw}}(\mathbb{Q}_\Sigma/\mathbb{Q}_{\infty}, \overline{\rho}_{i}(k-r))$$
in the above diagram (\ref{eqn:commutative-diagram-reduction})
is independent of the choice of $f^{\Sigma_0}$ in $S_k(\Gamma_1(N^{\Sigma_0}), \mathcal{O}_\pi )_{\mathfrak{m}^{\Sigma_0}} $.
\end{prop}

Due to Proposition \ref{prop:independent_of_f},  the congruence between cocycles
$$\delta^{\pm}_{f^{\Sigma_0}} \equiv u^{\pm} \cdot \delta^{\pm}_{g^{\Sigma_0}} \pmod{\pi} $$
in $\overline{\rho}$
yields the congruence between zeta elements
\begin{equation} \label{eqn:congruence-zeta-elements-same-levels}
\overline{\mathbf{z}_{\mathrm{Kato}}(f^{\Sigma_0},i, k-r)} = \overline{u}^{\pm} \cdot \overline{\mathbf{z}_{\mathrm{Kato}}(g^{\Sigma_0},i, k-r)}
\end{equation}
in $\mathrm{H}^1_{\mathrm{Iw}}(\mathbb{Q}_S/\mathbb{Q}_{\infty}, \overline{\rho}_{f, i}(k-r)) $
where $\overline{u}^{\pm}$ is the mod $\pi$ reduction of $u^{\pm} \in \mathcal{O}^\times_\pi$ and its sign coincides with that of $ (-1)^i$.


\subsection{Congruences between zeta elements: the different levels} \label{subsec:different-levels}
We discuss the congruence between two eigenforms of different levels.

\subsubsection{}
We have $V_{f} \simeq V_{f^{\Sigma_0}}$ as Galois representations, and we write
$T_{f} = \mathcal{O}_\pi \delta^+_f + \mathcal{O}_\pi  \delta^-_f$ and
$T_{f^{\Sigma_0}} = \mathcal{O}_\pi  \delta^+_{f^{\Sigma_0}} + \mathcal{O}_\pi  \delta^-_{f^{\Sigma_0}}$.
Let
\begin{align*}
\mathbf{z}_{\mathrm{Kato}}(f, k-r) & \in \mathbf{H}^1(j_*  T_{f}(k-r)), \\
\mathbf{z}_{\mathrm{Kato}}(f^{\Sigma_0}, k-r) & \in \mathbf{H}^1(j_*  T_{f^{\Sigma_0}}(k-r))
\end{align*}
be Kato's zeta elements associated to $\delta_f$ and  $\delta_{f^{\Sigma_0}}$ under
the integral zeta morphisms (\ref{eqn:zeta-morphism-single-modular-form-integral}) for $f$ and $f^{\Sigma_0}$, respectively.
We write
$$\mathbf{z}^{\Sigma_0}_{\mathrm{Kato}}(f, k-r) = \left( \prod_{\ell \in \Sigma_0} \mathcal{E}_\ell(f^*, r)^\iota \right) \cdot \mathbf{z}_{\mathrm{Kato}}(f, k-r).$$

We first reduce the comparison between $\mathbf{z}_{\mathrm{Kato}}(f^{\Sigma_0}, k-r)$ and $\mathbf{z}^{\Sigma_0}_{\mathrm{Kato}}(f, k-r)$ to the comparison between $\delta_f$ and  $\delta_{f^{\Sigma_0}}$.
The latter comparison enhances the isomorphism $V_{f} \simeq V_{f^{\Sigma_0}}$ integrally
and follows from Ihara's lemma.

\subsubsection{}
We first study two Kato's zeta elements
$\mathbf{z}_{\mathrm{Kato}}(f^{\Sigma_0}, k-r)$ and $\mathbf{z}^{\Sigma_0}_{\mathrm{Kato}}(f, k-r)$
in the $p$-inverted Iwasawa cohomology $\mathbf{H}^1(j_*  V_{f}(k-r)) \simeq \mathbf{H}^1(j_*  V_{f^{\Sigma_0}}(k-r)) \simeq \Lambda \otimes \mathbb{Q}_p$ (Theorem \ref{thm:kato-iwasawa-cohomologies}.(1)).
Since $p$ is inverted, the $p$-power direction of the Euler system relation and the interpolation formula of Kato's zeta element for every finite order character $\chi$ on $\mathrm{Gal}(\mathbb{Q}(\zeta_{p^\infty})/\mathbb{Q})$ detects the zeta element uniquely in the $p$-inverted Iwasawa cohomology.

Let $m$ be the square-free product of the primes in $\Sigma_0$.
Kato's explicit reciprocity law \cite[Theorem 12.5.(1)]{kato-euler-systems}\footnote{The referee pointed out that the same explicit reciprocity law argument applies to $f^{\Sigma_0}$ although newforms are only mentioned in Kato's paper.} implies that
\begin{align} \label{eqn:kato-explicit-reciprocity-law}
\begin{split}
\sum_{a \in \mathrm{Gal}(\mathbb{Q}(\zeta_{p^n})/\mathbb{Q})} \mathrm{exp}^* \circ \mathrm{loc}_p \left( \sigma_a \left( \mathbf{z}_{\mathrm{Kato}}(f^{\Sigma_0}, k-r) \right) \right) \cdot \chi(a)
& = (2 \pi i )^{k-r-1} \cdot L^{(mp)}(f^*,\chi, r) \cdot \delta^{\pm}_{f^{\Sigma_0}} , \\
\sum_{a \in \mathrm{Gal}(\mathbb{Q}(\zeta_{p^n})/\mathbb{Q})} \mathrm{exp}^* \circ \mathrm{loc}_p \left( \sigma_a \left( \mathbf{z}^{\Sigma_0}_{\mathrm{Kato}}(f, k-r) \right) \right)\cdot \chi(a)
& = (2 \pi i )^{k-r-1} \cdot L^{(mp)}(f^*,\chi, r) \cdot \delta^{\pm}_{f} .
\end{split}
\end{align}
for every character $\chi : \mathrm{Gal}(\mathbb{Q}(\zeta_{p^n})/\mathbb{Q}) \to \mathbb{C}^\times$
where $L^{(mp)}(f^*,\chi, r)$ is the $mp$-imprimitive $L$-value of $f^*$ at $s= r$ twisted by $\chi$, and  $\pm$ is the sign of $(-1)^{k-r-1} \cdot \chi(-1)$.
Here, ``$mp$-imprimitive" means that the Euler factors at primes dividing $mp$ are removed.

The explicit formula (\ref{eqn:kato-explicit-reciprocity-law}) implies that for every $0 \leq i \leq p-2$,
the zeta elements $\mathbf{z}_{\mathrm{Kato}}(f^{\Sigma_0},  i , k-r)$ and $\mathbf{z}^{\Sigma_0}_{\mathrm{Kato}}(f, i, k-r)$ differ only by an element of $F^\times_\pi$ in $\mathbb{H}^1(j_*  V_{f, i}(k-r))$.
In particular, the difference is exactly the ratio between
$\delta^{\pm}_{f^{\Sigma_0}}$ and $\delta^{\pm}_{f}$ in $V^{\pm}_f$
where $\pm$ is the sign of $(-1)^i$.

We now prove the ratio is a $p$-adic unit and this is where Ihara's lemma is needed.
\subsubsection{}
We follow the convention of \cite[$\S$1.7.3]{diamond-flach-guo}.
For positive integers $m$ dividing $N^{\Sigma_0}/N$, we define the morphism
$$\epsilon_m : S_{k}(\Gamma_1(N), \psi) \to  S_{k}(\Gamma_1(N^{\Sigma_0}), \psi)$$
defined by the double coset operator
$$m^{-1} \left[ U_0(N^{\Sigma_0})  \begin{pmatrix}
m^{-1} & 0 \\ 0 & 1
\end{pmatrix} U_0(N)  \right]$$
where $U_0(N)$ is the subgroup of the adelic points of $\mathrm{GL}_2$ corresponding to the $\Gamma_0(N)$-level structure.

Let $\phi_m$ be the endomorphism of $S_{k}(\Gamma_1(N), \psi)$ defined by
$\phi_1 = 1$, $\phi_\ell =  - T_\ell$, $\phi_{p^2} = \psi(p)p^{k-1}$, and $\phi_{m_1 \cdot m_2} = \phi_{m_1} \cdot \phi_{m_2}$ if $(m_1 , m_2)=1$ where $\ell$ is a prime dividing $N^{\Sigma_0}/N$.

We define
$$\gamma^{\Sigma_0} = \sum_m \epsilon_m \cdot \phi_m : S_{k}(\Gamma_1(N), \psi) \to  S_{k}(\Gamma_1(N^{\Sigma_0}), \psi) $$
where $m$ runs over the divisors of $N^{\Sigma_0}/N$.
Then we have $\gamma^{\Sigma_0}( f ) = f^{\Sigma_0}$ as in \cite[Proposition 1.4.(a)]{diamond-flach-guo}.

\begin{thm}[Ihara's lemma] \label{thm:ihara-lemma}
Assume that $2 \leq k \leq p-1$ and $p$ does not divide $N$.
If $\overline{\rho}$ is irreducible, then the map $\gamma^{\Sigma_0}$ induces an isomorphism
\[
\xymatrix{
\gamma^{\Sigma_0} : V_{k, \mathbb{Z}_p} (Y_1(N))_{\mathfrak{m}}[\wp_f] \ar[r]^-{\simeq} & V_{k, \mathbb{Z}_p} (Y_1(N^{\Sigma_0}))_{\mathfrak{m}^{\Sigma_0}}[\wp_{f^{\Sigma_0}}]
}
\]
where
 $\mathfrak{m}$
and  $\mathfrak{m}^{\Sigma_0}$ are maximal ideals of
$\mathbb{T}(N)$ and
$\mathbb{T}(N^{\Sigma_0})$
 corresponding to $f$ and $f^{\Sigma_0}$, respectively, and
 $\wp_f$ and $\wp_{f^{\Sigma_0}}$ are the height one prime ideals of $\mathbb{T}(N)$ and
$\mathbb{T}(N^{\Sigma_0})$
 corresponding to $f$ and $f^{\Sigma_0}$, respectively.
\end{thm}
\begin{proof}
See \cite[Proposition 1.4.(c)]{diamond-flach-guo}.
\end{proof}
\begin{cor} \label{cor:ihara-lemma}
Under the same assumptions in Theorem \ref{thm:ihara-lemma}, we have an isomorphism
\[
\xymatrix{
\gamma^{\Sigma_0} : T_f = V_{k, \mathbb{Z}_p} (Y_1(N))_{\mathfrak{m}} \otimes \mathbb{T}(N)_{\mathfrak{m}} / \wp_f \ar[r]^-{\simeq} & V_{k, \mathbb{Z}_p} (Y_1(N^{\Sigma_0}))_{\mathfrak{m}^{\Sigma_0}} \otimes \mathbb{T}(N^{\Sigma_0})_{\mathfrak{m}^{\Sigma_0}} /\wp_{f^{\Sigma_0}} = T_{f^{\Sigma_0}}
}
\]
sending $\delta^{\pm}_f$ to $v^{\pm} \cdot \delta^{\pm}_{f^{\Sigma_0}}$
where $v^{\pm} \in \mathcal{O}^\times_\pi$.
\end{cor}
\begin{proof}
It follows from the combination of Theorem \ref{thm:ihara-lemma} and the existence of integral perfect parings on
 $V_{k, \mathbb{Z}_p} (Y_1(N))_{\mathfrak{m}}$ and $V_{k, \mathbb{Z}_p} (Y_1(N^{\Sigma_0}))_{\mathfrak{m}^{\Sigma_0}}$ as in \cite[Corollary 1.6]{diamond-flach-guo}. More specifically,
\begin{align*}
V_{k, \mathbb{Z}_p} (Y_1(N))_{\mathfrak{m}} \otimes \mathbb{T}(N)_{\mathfrak{m}} / \wp_f
& \simeq
\mathrm{Hom} ( V_{k, \mathbb{Z}_p} (Y_1(N))_{\mathfrak{m}}[\wp_f] , \mathcal{O}_\pi(1-k)) & \textrm{\cite[Corollary 1.6]{diamond-flach-guo}} \\
& \simeq
\mathrm{Hom} (  V_{k, \mathbb{Z}_p} (Y_1(N^{\Sigma_0}))_{\mathfrak{m}^{\Sigma_0}}[\wp_{f^{\Sigma_0}}] , \mathcal{O}_\pi(1-k))
& \textrm{Theorem \ref{thm:ihara-lemma}} \\
& \simeq
V_{k, \mathbb{Z}_p} (Y_1(N^{\Sigma_0}))_{\mathfrak{m}^{\Sigma_0}}  \otimes \mathbb{T}(N^{\Sigma_0})_{\mathfrak{m}^{\Sigma_0}} /\wp_{f^{\Sigma_0}} & \textrm{\cite[Corollary 1.6]{diamond-flach-guo}} .
\end{align*}
\end{proof}

By combining (\ref{eqn:kato-explicit-reciprocity-law}) with Corollary \ref{cor:ihara-lemma}, we have congruence between zeta elements
\begin{equation} \label{eqn:congruence-zeta-elements-different-levels}
\overline{\mathbf{z}^{\Sigma_0}_{\mathrm{Kato}}(f, i, k-r)} = \overline{v}^{\pm} \cdot \overline{\mathbf{z}_{\mathrm{Kato}}(f^{\Sigma_0}, i, k-r)}
\end{equation}
in $\mathrm{H}^1_{\mathrm{Iw}}(\mathbb{Q}_{\Sigma} / \mathbb{Q}(\zeta_{p^\infty}), \overline{\rho}_{f}(k-r))$
where $\overline{v}^{\pm} \in \mathbb{F}^\times_\pi$ and the sign of $v^{\pm}$ coincides with that of $(-1)^i$.

\subsection{Putting it all together}
We have the following congruences
\begin{align} \label{eqn:compare-zeta-elements}
\begin{split}
  \overline{ \mathbf{z}^{\Sigma_0}_{\mathrm{Kato}}(f, i,k-r)  }  & = \overline{u}^{\pm}_1 \cdot \overline{ \mathbf{z}_{\mathrm{Kato}}(f^{\Sigma_0}, i,k-r) }  \\
& = \overline{u}^{\pm}_2 \cdot \overline{ \mathbf{z}_{\mathrm{Kato}}(g^{\Sigma_0}, i,k-r) } \\
&  = \overline{u}^{\pm}_3 \cdot \overline{ \mathbf{z}^{\Sigma_0}_{\mathrm{Kato}}(g, i,k-r) }
\end{split}
\end{align}
in $\mathrm{H}^1_{\mathrm{Iw}}(\mathbb{Q}_{\Sigma}/\mathbb{Q}(\zeta_{p^\infty}), \overline{\rho}(k-r) ) $
where $\overline{u}^{\pm}_1$, $\overline{u}^{\pm}_2$, and $\overline{u}^{\pm}_3$ lie in $\mathbb{F}^\times_\pi$.
The first and the third equalities follow from (\ref{eqn:congruence-zeta-elements-different-levels}), and
the second equality follows from (\ref{eqn:congruence-zeta-elements-same-levels}).

Thus, the following consequence is immediate.
\begin{thm} \label{thm:congruences-kato-elements}
If $\Sigma_0$ contains the primes dividing $N_f$ and $N_g$, then
 $\overline{\mathbf{z}^{\Sigma_0}_{\mathrm{Kato}}(f, i, k- r)}$ and $\overline{\mathbf{z}^{\Sigma_0}_{\mathrm{Kato}}(g, i, k- r)}$
coincide in $\mathrm{H}^1_{\mathrm{Iw}}(\mathbb{Q}_{\Sigma}/\mathbb{Q}_\infty, \overline{\rho}_{f,i}(k-r) )$.
In particular, we have
\begin{align} \label{eqn:index-zeta-elements}
\begin{split}
& \left[ \mathrm{H}^1_{\mathrm{Iw}}(\mathbb{Q}_{\Sigma}/\mathbb{Q}_\infty, \overline{\rho}_{f,i}(k-r) ) : \Lambda_i / \pi_f \Lambda_i \cdot \overline{\mathbf{z}^{\Sigma_0}_{\mathrm{Kato}}(f, i, k- r)} \right] \\
= \ &
\left[ \mathrm{H}^1_{\mathrm{Iw}}(\mathbb{Q}_{\Sigma}/\mathbb{Q}_\infty, \overline{\rho}_{g,i}(k-r) ) : \Lambda_i / \pi_g \Lambda_i \cdot \overline{\mathbf{z}^{\Sigma_0}_{\mathrm{Kato}}(g, i, k- r)} \right]
.
\end{split}
\end{align}
\end{thm}
\begin{rem}
If one of (\ref{eqn:index-zeta-elements}) is finite (``$\mu=0$"), then the other is also finite.
In this case, the index (\ref{eqn:index-zeta-elements}) contains the information of the $\lambda$-invariant of
$\mathbb{H}^1(j_* T_{f,i}(k-r)) / \mathbf{z}^{\Sigma_0}_{\mathrm{Kato}}(f, i, k- r)$ and the size of  $\mathrm{H}^2_{\mathrm{Iw}}(\mathbb{Q}_{\Sigma}/\mathbb{Q}_{\infty},T_{f,i}(k-r) )[\pi]$ simultaneously; thus, it seems difficult to obtain the formula on $\lambda$-invariants under congruences as in the literature without having the non-existence of finite Iwasawa submodule of $\mathrm{H}^2_{\mathrm{Iw}}(\mathbb{Q}_{\Sigma}/\mathbb{Q}_{\infty},T_{f,i}(k-r) )[\pi]$.
\end{rem}
The following corollary proves Theorem \ref{thm:main-theorem}.(1). The statement is independent of the choice of $\Sigma_0$.
\begin{cor} \label{cor:mu-invariants-vanish-in-families}
If $\overline{\mathbf{z}_{\mathrm{Kato}}(f_0, i, k- r)}$ is non-zero for some $f_0 \in S_k(\overline{\rho})$, then $\overline{\mathbf{z}_{\mathrm{Kato}}(f, i, k- r)}$ is non-zero for all $f \in S_k(\overline{\rho})$.
\end{cor}
\begin{proof}
It follows from the equality of $\mu$-invariants in (\ref{eqn:iwasawa-invariants-kato-elements}) and Theorem \ref{thm:congruences-kato-elements}.
\end{proof}

\section{The $\mathrm{H}^2$-side} \label{sec:H^2}
Recall that $S$ is a finite set of places of $\mathbb{Q}$ containing the primes dividing $pN$, and $S' = S \setminus \lbrace p \rbrace$.
\begin{prop}[Kurihara] \label{prop:localization-sequence}
\begin{enumerate}
\item The canonical mapping
\[
\xymatrix@R=0em{
\mathbf{H}^1(j_*T_{f}(k-r)) \ar[r]^-{\simeq} & \mathrm{H}^1_{\mathrm{Iw}}(\mathbb{Q}_S/\mathbb{Q}(\zeta_{p^\infty}), T_f(k-r))
}
\]
is an isomorphism.
\item The sequence
\[
	\begin{tikzcd}
		0 \ar{r} & \mathbf{H}^{2}( j_* T_f(k-r) ) \ar{r} \ar[phantom, ""{coordinate, name=Z}]{d} & \mathrm{H}^{2}_{\mathrm{Iw}}( \mathbb{Q}_S/\mathbb{Q}(\zeta_{p^\infty}), T_{f}(k-r))  \ar[rounded corners, to path={ -- ([xshift=2ex]\tikztostart.east) |- (Z) [near end]\tikztonodes -| ([xshift=-2ex]\tikztotarget.west) -- (\tikztotarget)}]{dl}\\
		& \displaystyle{ \bigoplus_{\ell \in S'} \bigoplus_{\eta \mid \ell} \mathrm{H}^2_{\mathrm{Iw}}(\mathbb{Q}(\zeta_{p^\infty})_\eta, T_{f}(k-r)) } \ar{r} & 0
	\end{tikzcd}
\]
is an exact sequence of $\Lambda$-modules.
\end{enumerate}
\end{prop}
\begin{rem}
\begin{enumerate}
\item
The first statement is given in \cite[Proposition 7.1.i)]{kobayashi-thesis} and \cite[$\S$6]{kurihara-invent} using the localization sequence of \'{e}tale cohomology.
We give a direct proof of both statements using the same method.
\item
Proposition \ref{prop:localization-sequence}.(2) is an analogue of \cite[Proposition 2.1]{gv}.
However, the $\Lambda$-torsionness of $\mathbf{H}^{2}( j_* T_f(k-r) )$ (Theorem \ref{thm:kato-iwasawa-cohomologies}.(3)) is not required in this setting.
\end{enumerate}
\end{rem}

Recall the following results by Jannsen~\cite{jannsen-continuous-etale} concerning continuous étale cohomology, which will be needed.
Let \(X\) be a scheme.
Let \(\Sh(X_\et)^{\N}\) be the category of inverse system of étale sheaves on the small étale site of \(X\).
The category \(\Sh(X_\et)^{\N}\) is abelian with enough injectives.
If \((\Fc_n) \in \Sh(X_\et)^{\N}\), following Jannsen, we denote by \(\HH^i_\et(X,(\Fc_n))\) the \(i\)-th derived functor of the left exact functor
\[
	\begin{split}
		\Sh(X_\et)^{\N} & \rightarrow \Ab \\
		(\Fc_n) & \mapsto \varprojlim \HH^0_\et(X,\Fc_n).
	\end{split}
\]
Recall that there are short exact sequences
\begin{equation} \label{eq:continuous}
	0 \rightarrow {\varprojlim}^1 \HH^{i-1}_\et(X,\Fc_n) \rightarrow \HH^i_\et(X,(\Fc_n)) \rightarrow \varprojlim \HH^i_\et(X,\Fc_n) \rightarrow 0,
\end{equation}
where by convention \(\HH^{-1}_\et(X,\Fc_n) = 0\).
In particular, if the systems \(( \HH^{i-1}_\et(X,\Fc_n))\) satisfies the Mittag--Leffler condition, then
\[
	\HH^i_\et(X,(\Fc_n))  \simeq  \varprojlim \HH^i_\et(X,\Fc_n).
\]
Moreover, if \(f : X \to Y\) is a morphism of schemes, then~\cite[(3.10)]{jannsen-continuous-etale} there exists a spectral sequence
\[
	E_2^{a,b} = \HH^a_\et(Y,(R^b f_\ast \Fc_n)) \Rightarrow \HH^{a+b}_\et(X,(\Fc_n)).
\]

Let \(F\) be a number field, and let \(\Oc\) be its ring of integers.
Recall that \(p\) is an odd prime so that the local Galois cohomology at the infinite place can be ignored.
Also, \(S\) is a finite set of finite primes containing \(p\).
We set \(X = \Spec (\Oc[1/S])\) and \(Y = \Spec(\Oc[1/p])\), and \(j : X \rightarrow Y\) the natural map.

Let \(T\) be a free \(\Zp\)-module of finite rank equipped with a continuous and \(\Zp\)-linear action by \(G_{F,S} = \pi_\et^1(X)\). 
Then, each \(T/p^nT\) defines an \'{e}tale sheaves on \(X\), which we again denote by \(T/p^n T\), and \(\HH^i_\et(X,T/p^nT) \simeq \HH^i(G_{F,S},T/p^nT)\) (see~\cite[II Proposition~2.9]{milne-adt}).
By Tate (see~\cite[\S~8.3]{neukirch2} and~\cite[\S~2]{tate_k_2_galois_cohomology}), the groups \(\HH^i(G_{F,S},T/p^nT)\) are finite, and hence the systems \((\HH^i(G_{F,S},T/p^nT))\) satisfies the Mittag--Leffler condition, and therefore, there are isomorphisms
\[
	\varprojlim \HH^i(G_{F,S},T/p^nT) \simeq \HH^i(G_{F,S},T).
\]
It follows that \(\HH^i_\et(X,(T/p^nT))\) is isomorphic to
\[
	 \HH^i_\et(X,T) = \varprojlim \HH^i_\et(X,T/p^nT) \simeq \varprojlim \HH^i(G_{F,S},T/p^nT) \simeq \HH^i(G_{F,S},T).
\]
To compute \(\HH^p_\et(Y,(R^q j_\ast T/p^nT))\), by the short exact sequence~\eqref{eq:continuous}, it is enough to consider \(\HH^p_\et(Y,R^q j_\ast T/p^nT)\), whose computation is standard (see~\cite[Exp. VIII]{SGAIV:2} and~\cite[III]{soule79}).
For \(q > 0\),  and with \(\kappa(v)\) the residue field at a place \(v\) of \(F\) and \(F_v^\mathrm{ur}\) the maximal unramified extension of \(F_v\), we have
\[
	\HH^p_\et(Y,R^q j_\ast T/p^nT) \simeq \oplus_{v \mid \ell, \ell \in S^\prime} \HH^p(\kappa(v), \HH^q(F_v^\mathrm{ur},T/p^n T)),
\]
which again by Tate and the Hochschild--Serre spectral sequence, are finite groups, and hence \((\HH^p_\et(Y,R^q j_\ast T/p^nT))\) satisfies the Mittag--Leffler condition.
The finiteness of the groups \(\HH^i_\et(X,T/p^n T)\) and \(\HH^p_\et(Y,R^q j_\ast T/p^nT)\) for \(q>0\) implies, by the Leray spectral sequence \(\HH^a_\et(Y,R^b j_\ast T/p^n T) \Rightarrow \HH^{a+b}_\et(X,T/p^nT)\), that the groups \(\HH^i_\et(Y,j_\ast T/p^nT)\) are finite, and thus the system \((\HH^i_\et(Y,j_\ast T/p^nT))\) satisfies the Mittag--Leffler condition.
Therefore, we obtain that \(\HH^i_\et(Y,(R^q j_\ast T/p^nT))\) is isomorphic to
\[
	\HH^i_\et(Y,j_\ast T) = \varprojlim \HH^i_\et(Y,j_\ast T/p^nT)
\]
if \(q = 0\), and to
\[
	\begin{split}
		\HH^p_\et(Y,R^q j_\ast T) & = \varprojlim \HH^p_\et(Y,R^q j_\ast T/p^nT)\\
		& \simeq \oplus_{v \mid \ell, \ell \in S^\prime} \varprojlim  \HH^p(\kappa(v), \HH^q(F_v^\mathrm{ur},T/p^n T)) \\
		& \simeq \oplus_{v \mid \ell, \ell \in S^\prime}\HH^p(\kappa(v), \HH^q(F_v^\mathrm{ur},T))
	\end{split}
 \]
if \(q > 0\).

\begin{proof}[Proof of Proposition~\ref{prop:localization-sequence}]
We use $j$ instead of $j_n$ here for notational convenience.
By the previous discussion, the low-degree terms of the Leray spectral sequence
$$E^{a,b}_2 = \mathrm{H}^a_{\mathrm{\acute{e}t}}(\mathbb{Z}[\zeta_{p^n}, 1/p], R^b j_*T_f(k-r)) \Rightarrow
\mathrm{H}^{a+b} (\mathbb{Q}_S/\mathbb{Q}(\zeta_{p^n}), T_f(k-r)) $$
induces the following localization exact sequence in \'{e}tale cohomology
\begin{align*}
0
& \to \mathrm{H}^1_{\mathrm{\acute{e}t}}(\mathbb{Z}[\zeta_{p^n}, 1/p], j_*T_f(k-r)) \\
& \to \mathrm{H}^1(\mathbb{Q}_S/\mathbb{Q}(\zeta_{p^n}), T_f(k-r)) \\
& \to \bigoplus_{\ell \in S'} \bigoplus_{\eta \mid \ell}  \mathrm{H}^0_{\mathrm{\acute{e}t}}( \kappa(\eta), \mathrm{H}^1_{\mathrm{\acute{e}t}}(\mathbb{Q}(\zeta_{p^n})^{\mathrm{ur}}_\eta, T_f(k-r)) ) \\
& \to \mathrm{H}^2_{\mathrm{\acute{e}t}}(\mathbb{Z}[\zeta_{p^n}, 1/p], j_*T_f(k-r)) \\
& \to \mathrm{ker} \left( \mathrm{H}^2(\mathbb{Q}_S/\mathbb{Q}(\zeta_{p^n}), T_f(k-r)) \to \bigoplus_{\ell \in S'} \bigoplus_{\eta \mid \ell}  \mathrm{H}^0_{\mathrm{\acute{e}t}}( \kappa(\eta), \mathrm{H}^2_{\mathrm{\acute{e}t}}(\mathbb{Q}(\zeta_{p^n})^{\mathrm{ur}}_\eta, T_f(k-r)) ) \right) \\
& \to \bigoplus_{\ell \in S'} \bigoplus_{\eta \mid \ell}  \mathrm{H}^1_{\mathrm{\acute{e}t}}( \kappa(\eta), \mathrm{H}^1_{\mathrm{\acute{e}t}}(\mathbb{Q}(\zeta_{p^n})^{\mathrm{ur}}_\eta, T_f(k-r)) ) \\
& \to \mathrm{H}^3_{\mathrm{\acute{e}t}}(\mathbb{Z}[\zeta_{p^n}, 1/p], j_*T_f(k-r))
\end{align*}
where $\kappa(\eta)$ is the residue field of $\mathbb{Z}[\zeta_{p^n}]$ at $\eta$ and $F^{\mathrm{ur}}$ is the maximal unramified extension of a local field $F$. See Remark \ref{rem:localization-exact-sequences} below for the references on localization exact sequences.

Since the cohomological dimension of $\kappa(\eta)$ is 1,
 the $p$-cohomological dimension of $\mathbb{Q}(\zeta_{p^n})^{\mathrm{ur}}_\eta$ is $\leq 1$, and the $p$-\'{e}tale cohomological dimension of \(\mathbb{Z}[\zeta_{p^n}, 1/p]\) is \(\leq 2\) (see~\cite[Lemma~5]{schneider-height-2}), the above exact sequence becomes
\begin{align*}
0
& \to \mathrm{H}^1_{\mathrm{\acute{e}t}}(\mathbb{Z}[\zeta_{p^n}, 1/p], j_*T_f(k-r)) \\
& \to \mathrm{H}^1(\mathbb{Q}_S/\mathbb{Q}(\zeta_{p^n}), T_f(k-r)) \\
& \to \bigoplus_{\ell \in S'} \bigoplus_{\eta \mid \ell}  \mathrm{H}^0_{\mathrm{\acute{e}t}}( \kappa(\eta), \mathrm{H}^1_{\mathrm{\acute{e}t}}(\mathbb{Q}(\zeta_{p^n})^{\mathrm{ur}}_\eta, T_f(k-r)) ) \\
& \to \mathrm{H}^2_{\mathrm{\acute{e}t}}(\mathbb{Z}[\zeta_{p^n}, 1/p], j_*T_f(k-r)) \\
& \to  \mathrm{H}^2(\mathbb{Q}_S/\mathbb{Q}(\zeta_{p^n}), T_f(k-r))  \\
& \to \bigoplus_{\ell \in S'} \bigoplus_{\eta \mid \ell}  \mathrm{H}^2_{\mathrm{\acute{e}t}}(\mathbb{Q}(\zeta_{p^n})_\eta, T_f(k-r))  \\
& \to 0 .
\end{align*}
For $\ell \in S'$, we have
$$\varprojlim_{n}  \mathrm{H}^0_{\mathrm{\acute{e}t}}( \kappa(\eta_n), \mathrm{H}^1_{\mathrm{\acute{e}t}}(\mathbb{Q}(\zeta_{p^n})^{\mathrm{ur}}_{\eta_n}, T_f(k-r)) ) = 0$$
where $\eta_n$ is a prime of $\mathbb{Q}(\zeta_{p^n})$ dividing $\ell$. Thus, we obtain the conclusion.
\end{proof}
\begin{rem} \label{rem:localization-exact-sequences}
For the general description of the localization exact sequences for \'{e}tale cohomology, we refer to \cite[Theorem 7.4.4.(3)]{lei-fu-etale-cohomology}, \cite[\href{https://stacks.math.columbia.edu/tag/09XP}{Section 09XP}]{stacks-project}, and \cite[III. Proposition 1]{soule79}.
The application of the localization exact sequence to this Iwasawa-theoretic setting can be found in \cite[\S6]{kurihara-invent} and \cite{kato-euler-systems}. In particular, Kato implicitly used this sequence to compare various Selmer groups (see \cite[\S17.10]{kato-euler-systems} for example).
\end{rem}
Since $\mathbb{H}^{2}(j_* T_{f,i}(k-r))$ is a finitely generated \emph{torsion} $\Lambda_i$-module (Theorem \ref{thm:kato-iwasawa-cohomologies}.(3)) and the local $\mathrm{H}^2$'s are also finitely generated torsion $\Lambda_i$-modules as in Corollary \ref{cor:local-H^2}, $\mathrm{H}^{2}_{\mathrm{Iw}}( \mathbb{Q}_S/\mathbb{Q}_{\infty}, T_{f,i}(k-r))$ is also a finitely generated torsion $\Lambda_i$-module due to  Proposition \ref{prop:localization-sequence}.(2).

By Proposition \ref{prop:localization-sequence}.(2) and the multiplicative property of characteristic ideals (c.f. \cite[Proposition 1 in Appendix, Page 104--105]{coates-sujatha}), we have
{ \small
\begin{equation} \label{eqn:multiplicative-char-ideals}
\mathrm{char}_{\Lambda_i} \left(  \mathrm{H}^{2}_{\mathrm{Iw}}( \mathbb{Q}_S/\mathbb{Q}_{\infty}, T_{f, i}(k-r)) \right) =
\mathrm{char}_{\Lambda_i} \left( \mathbb{H}^{2}( j_* T_{f, i}(k-r) ) \right)  \cdot
\displaystyle{ \prod_{\ell \in S'}  } \mathrm{char}_{\Lambda_i} \left( \bigoplus_{\eta \mid \ell} \mathrm{H}^2_{\mathrm{Iw}}(\mathbb{Q}_{\infty, \eta}, T_{f,i}(k-r))  \right) .
\end{equation}
}
\begin{prop} \label{prop:iwasawa-invariants-H^2} $ $
\begin{enumerate}
\item $\mu \left( \mathrm{H}^{2}_{\mathrm{Iw}}( \mathbb{Q}_S/\mathbb{Q}_{\infty}, T_{f, i}(k-r)) \right) = \mu \left( \mathbb{H}^{2}( j_* T_{f, i}(k-r) )  \right)$.
\item
 $\lambda \left( \mathrm{H}^{2}_{\mathrm{Iw}}( \mathbb{Q}_S/\mathbb{Q}_{\infty}, T_{f, i}(k-r)) \right) = \lambda \left( \mathbb{H}^{2}( j_* T_{f, i}(k-r) )  \right) +  \displaystyle{ \sum_{\ell \in S'}  } \lambda \left( \bigoplus_{\eta \mid \ell} \mathrm{H}^2_{\mathrm{Iw}}(\mathbb{Q}_{\infty, \eta}, T_{f,i}(k-r))  \right) $.
\end{enumerate}
\end{prop}
\begin{proof}
Both the statements immediately follow from (\ref{eqn:multiplicative-char-ideals}) and Corollary \ref{cor:local-H^2}.
\end{proof}
\begin{rem} \label{rem:iwasawa-invariants-H^2}
Note that we do \emph{not} require any $\mu=0$ assumption or any statement about the non-existence of finite $\Lambda_i$-submodules for Proposition \ref{prop:iwasawa-invariants-H^2}.(2).
Thus, Proposition \ref{prop:iwasawa-invariants-H^2} partially generalizes \cite[Proposition 2.8]{gv} and \cite[Theorem 4.3.4]{epw} by removing these conditions in some sense.
\end{rem}
For any $\Sigma_0 \subseteq S' = S \setminus \lbrace p, \infty \rbrace$, we define the $\Sigma_0$-imprimitive version of $\mathbb{H}^{2}(j_* T_{f,i}(k-r))$ by the following exact sequence
\[
	\begin{tikzcd}
		0 \ar{r} & \mathbb{H}^{2, \Sigma_0}(j_* T_{f,i}(k-r)) \ar{r} \ar[phantom, ""{coordinate, name=Z}]{d} & \mathrm{H}^{2}_{\mathrm{Iw}}( \mathbb{Q}_S/\mathbb{Q}_{\infty}, T_{f,i} (k-r) )  \ar[rounded corners, to path={ -- ([xshift=2ex]\tikztostart.east) |- (Z) [near end]\tikztonodes -| ([xshift=-2ex]\tikztotarget.west) -- (\tikztotarget)}]{dl} \\
& \bigoplus_{\ell \in S' \setminus \Sigma_0} \bigoplus_{\eta \mid \ell} \mathrm{H}^2_{\mathrm{Iw}}(\mathbb{Q}_{\infty, \eta}, T_{f, i} (k-r) ) \ar{r} & 0.
	\end{tikzcd}
\]
Applying the same argument to $\mathbb{H}^{2, \Sigma_0}(j_* T_{f,i}(k-r))$, we obtain the following statements.
\begin{prop} \label{prop:iwasawa-invariants-H^2-imprimitive} $ $
\begin{enumerate}
\item $\mu \left( \mathrm{H}^{2}_{\mathrm{Iw}}( \mathbb{Q}_S/\mathbb{Q}_{\infty}, T_{f, i}(k-r)) \right) = \mu \left( \mathbb{H}^{2, \Sigma_0}( j_* T_{f, i}(k-r) )  \right)$.
\item
 $\lambda \left( \mathrm{H}^{2}_{\mathrm{Iw}}( \mathbb{Q}_S/\mathbb{Q}_{\infty}, T_{f, i}(k-r)) \right) = \lambda \left( \mathbb{H}^{2, \Sigma_0}( j_* T_{f, i}(k-r) )  \right) +  \displaystyle{ \sum_{\ell \in S' \setminus \Sigma_0}  } \lambda \left( \bigoplus_{\eta \mid \ell} \mathrm{H}^2_{\mathrm{Iw}}(\mathbb{Q}_{\infty, \eta}, T_{f,i}(k-r))  \right) $.
\end{enumerate}
\end{prop}

\section{The invariance of the difference of $\lambda$-invariants} \label{sec:connection}
In this section, we prove Theorem \ref{thm:main-theorem}.(2).
\begin{lem} \label{lem:key-lemma}
Let $X$ be a finitely generated torsion $\mathcal{O}_\pi \llbracket T \rrbracket$-module with $\mu(X) = 0$.
Then
$\dfrac{ \# X /\pi X }{ \#X[\pi] } = (\#\mathbb{F})^{\lambda(X)}$.
\end{lem}
\begin{proof}
Straightforward.
\end{proof}
Consider the exact sequence
\begin{equation} \label{eqn:mod-p-exact-seq}
	\begin{tikzcd}
		0 \ar{r} & \dfrac{ \mathrm{H}^1_{\mathrm{Iw}}(\mathbb{Q}_S/\mathbb{Q}_\infty, T_{f,i}(k-r) ) }{ \pi \mathrm{H}^1_{\mathrm{Iw}}( \mathbb{Q}_S/\mathbb{Q}_\infty, T_{f,i}(k-r) ) } \ar{r} \ar[phantom, ""{coordinate, name=Z}]{d} & \mathrm{H}^1_{\mathrm{Iw}}(\mathbb{Q}_S/\mathbb{Q}_\infty, \overline{\rho}_{f,i}(k-r))  \ar[rounded corners, to path={ -- ([xshift=2ex]\tikztostart.east) |- (Z) [near end]\tikztonodes -| ([xshift=-2ex]\tikztotarget.west) -- (\tikztotarget)}]{dl} \\
		& \mathrm{H}^2_{\mathrm{Iw}}(\mathbb{Q}_S/\mathbb{Q}_{\infty} ,T_{f,i}(k-r))[\pi] \ar{r} & 0 .
	\end{tikzcd}
\end{equation}
Since $p$ is odd and $\mathrm{Gal}(\mathbb{Q}_S/\mathbb{Q}_n)$ has $p$-cohomological dimension 2, we have an isomorphism
$$
\dfrac{
\mathrm{H}^{2}( \mathbb{Q}_S/\mathbb{Q}_{n}, T_{f,i} (k-r) )
}{
\pi \mathrm{H}^{2}( \mathbb{Q}_S/\mathbb{Q}_{n}, T_{f, i} (k-r) )
}
\simeq \mathrm{H}^{2}( \mathbb{Q}_S/\mathbb{Q}_{n}, \overline{\rho}_{f,i} (k-r) ) .$$
Taking the inverse limit, we have
\begin{equation} \label{eqn:mod-p-H2}
\dfrac{
\mathrm{H}^{2}_{\mathrm{Iw}}( \mathbb{Q}_S/\mathbb{Q}_{\infty}, T_{f,i} (k-r) )
}{
\pi \mathrm{H}^{2}_{\mathrm{Iw}}( \mathbb{Q}_S/\mathbb{Q}_{\infty}, T_{f, i} (k-r) )
}
\simeq \mathrm{H}^{2}_{\mathrm{Iw}}( \mathbb{Q}_S/\mathbb{Q}_{\infty}, \overline{\rho}_{f,i} (k-r) ) .
\end{equation}
Combining Lemma \ref{lem:key-lemma}, (\ref{eqn:mod-p-exact-seq}), and (\ref{eqn:mod-p-H2}) with $\mu \left( \mathrm{H}^{2}_{\mathrm{Iw}}( \mathbb{Q}_S/\mathbb{Q}_{\infty}, T_{f, i} (k-r) ) \right) = 0$,
we have
\begin{align} \label{eqn:mod-p-indices}
\left[
 \mathrm{H}^1_{\mathrm{Iw}}(\mathbb{Q}_S/\mathbb{Q}_\infty, \overline{\rho}_{f,i}(k-r)) :
\dfrac{ \mathrm{H}^1_{\mathrm{Iw}}(\mathbb{Q}_S/\mathbb{Q}_\infty, T_{f,i}(k-r) ) }{ \pi \mathrm{H}^1_{\mathrm{Iw}}( \mathbb{Q}_S/\mathbb{Q}_\infty, T_{f,i}(k-r) ) }
 \right] & = \dfrac{ \#\mathrm{H}^2_{\mathrm{Iw}}(\mathbb{Q}_S/\mathbb{Q}_{\infty}, \overline{\rho}_{f,i}(k-r)  ) }{
( \# \mathbb{F} )^{\lambda \left(  \mathrm{H}^2_{\mathrm{Iw}}(\mathbb{Q}_S/\mathbb{Q}_{\infty},T_{f,i}(k-r) )  \right)} } .
\end{align}

\begin{prop} \label{prop:comparison-iwasawa-invariants}
Let
$$\mathbf{z} \in \mathrm{H}^1_{\mathrm{Iw}}(\mathbb{Q}_S/\mathbb{Q}_\infty, T_{f,i}(k-r) )$$
 be a non-zero element and
$$\overline{\mathbf{z}} \in \dfrac{ \mathrm{H}^1_{\mathrm{Iw}}(\mathbb{Q}_S/\mathbb{Q}_\infty, T_{f,i}(k-r) ) }{ \pi \mathrm{H}^1_{\mathrm{Iw}}( \mathbb{Q}_S/\mathbb{Q}_\infty, T_{f,i}(k-r) ) } \subseteq \mathrm{H}^1_{\mathrm{Iw}}(\mathbb{Q}_S/\mathbb{Q}_\infty, \overline{\rho}_{f,i}(k-r) )$$
 the image of
$\mathbf{z}$ in $\mathrm{H}^1_{\mathrm{Iw}}(\mathbb{Q}_S/\mathbb{Q}_\infty, \overline{\rho}_{f,i}(k-r) )$.
Then we have the following statements.
\begin{enumerate}
\item The following statements are equivalent:
\begin{enumerate}
\item $\mu \left(  \mathrm{H}^1_{\mathrm{Iw}}(\mathbb{Q}_S/\mathbb{Q}_\infty, T_{f,i}(k-r) ) / \Lambda_i \mathbf{z} \right)  = 0$ and $\mu \left(  \mathrm{H}^2_{\mathrm{Iw}}(\mathbb{Q}_S/\mathbb{Q}_\infty, T_{f,i}(k-r) ) \right)  = 0$.
\item $\left[ \mathrm{H}^1_{\mathrm{Iw}}(\mathbb{Q}_S/\mathbb{Q}_\infty, \overline{\rho}_{f,i}(k-r) ) : \Lambda_i/\pi \Lambda_i \cdot \overline{\mathbf{z}} \right] < \infty $.
\end{enumerate}
\item If (1) holds, then
\begin{align} \label{eqn:invariant-under-congruences}
\begin{split}
&
\left[ \mathrm{H}^1_{\mathrm{Iw}}(\mathbb{Q}_S/\mathbb{Q}_\infty, \overline{\rho}_{f,i}(k-r) ) : \Lambda_i/\pi \Lambda_i \cdot \overline{\mathbf{z}} \right]  \times \left( \#\mathrm{H}^2_{\mathrm{Iw}}(\mathbb{Q}_S/\mathbb{Q}_{\infty}, \overline{\rho}_{f,i}(k-r)  ) \right)^{-1}\\
= \ &
( \# \mathbb{F} )^{ \lambda \left(  \mathbb{H}^1(j_* T_{f,i}(k-r)) / \Lambda_i \mathbf{z} \right) - \lambda \left( \mathrm{H}^2_{\mathrm{Iw}}(\mathbb{Q}_S/\mathbb{Q}_{\infty}, T_{f,i}(k-r) \right) } .
\end{split}
\end{align}
\end{enumerate}
\end{prop}
\begin{proof}
\begin{enumerate}
\item
Due to Theorem \ref{thm:kato-iwasawa-cohomologies}.(2) and Proposition \ref{prop:localization-sequence}, we have
 $$\mathrm{H}^1_{\mathrm{Iw}}(\mathbb{Q}_S/\mathbb{Q}_\infty, T_{f,i}(k-r) ) \simeq \Lambda_i .$$
Thus, the following statements are equivalent:
\begin{itemize}
\item $\left[ \mathrm{H}^1_{\mathrm{Iw}}(\mathbb{Q}_S/\mathbb{Q}_\infty, \overline{\rho}_{f,i}(k-r) ) : \Lambda_i/\pi \Lambda_i \cdot \overline{\mathbf{z}} \right] < \infty$,
\item
$\overline{\mathbf{z}} \neq 0  \textrm{ in } \mathrm{H}^1_{\mathrm{Iw}}(\mathbb{Q}_S/\mathbb{Q}_\infty, T_{f,i}(k-r) )/\pi \textrm{ and } \# \mathrm{H}^2_{\mathrm{Iw}}(\mathbb{Q}_S/\mathbb{Q}_\infty, T_{f,i}(k-r) )[\pi] < \infty $.
\end{itemize}
Note that the following statements are also equivalent:
\begin{itemize}
\item $\# \mathrm{H}^2_{\mathrm{Iw}}(\mathbb{Q}_S/\mathbb{Q}_\infty, T_{f,i}(k-r) )[\pi] < \infty$,
\item $\mu \left(  \mathrm{H}^2_{\mathrm{Iw}}(\mathbb{Q}_S/\mathbb{Q}_\infty, T_{f,i}(k-r) ) \right)  = 0$.
\end{itemize}
\item
Due to (\ref{eqn:mod-p-indices}), it suffices to compute the index of
$$\Lambda_i / \pi \Lambda_i \cdot \overline{\mathbf{z}}
\subseteq \dfrac{ \mathrm{H}^1_{\mathrm{Iw}}(\mathbb{Q}_S/\mathbb{Q}_\infty, T_{f,i}(k-r) ) }{ \pi \mathrm{H}^1_{\mathrm{Iw}}( \mathbb{Q}_S/\mathbb{Q}_\infty, T_{f,i}(k-r) ) } .$$
By Lemma \ref{lem:key-lemma},
we have
$$\dfrac{
\mathrm{H}^1_{\mathrm{Iw}}(\mathbb{Q}_S/\mathbb{Q}_\infty, T_{f,i}(k-r) ) / (\mathbf{z}, \pi)
}{
\left( \mathrm{H}^1_{\mathrm{Iw}}(\mathbb{Q}_S/\mathbb{Q}_\infty, T_{f,i}(k-r) ) / \Lambda_i\mathbf{z} \right) [\pi]
}
= \# \mathbb{F}^{\lambda \left( \mathrm{H}^1_{\mathrm{Iw}}(\mathbb{Q}_S/\mathbb{Q}_\infty, T_{f,i}(k-r) ) / \Lambda_i\mathbf{z} \right)}$$
and $\left( \mathrm{H}^1_{\mathrm{Iw}}(\mathbb{Q}_S/\mathbb{Q}_\infty, T_{f,i}(k-r) ) / \Lambda_i\mathbf{z} \right) [\pi] = 0$ since $\overline{\mathbf{z}} \neq 0$.
Thus, the index is
$$\# \mathbb{F}^{\lambda \left( \mathrm{H}^1_{\mathrm{Iw}}(\mathbb{Q}_S/\mathbb{Q}_\infty, T_{f,i}(k-r) ) / \Lambda_i\mathbf{z} \right)}$$
and we obtain the conclusion. Note that we have an isomorphism  $$ \mathbb{H}^1(j_* T_{f,i}(k-r)) \simeq  \mathrm{H}^1_{\mathrm{Iw}}(\mathbb{Q}_S/\mathbb{Q}_{\infty}, T_{f, i}(k-r)) $$ by Proposition \ref{prop:localization-sequence}.
\end{enumerate}
\end{proof}
\begin{rem}
Indeed, if $\mathbf{z} = \mathbf{z}_{\mathrm{Kato}}(f, i, k- r)$, then
$\mu \left(  \mathrm{H}^1_{\mathrm{Iw}}(\mathbb{Q}_S/\mathbb{Q}_\infty, T_{f,i}(k-r) ) / \Lambda_i \mathbf{z} \right)  = 0$ implies $\mu \left(  \mathrm{H}^2_{\mathrm{Iw}}(\mathbb{Q}_S/\mathbb{Q}_\infty, T_{f,i}(k-r) ) \right)  = 0$
by using Theorem \ref{thm:kato-divisibility}, Proposition \ref{prop:localization-sequence}, and Proposition \ref{prop:iwasawa-invariants-H^2}. Thus, Proposition \ref{prop:comparison-iwasawa-invariants}.(1).(a) can be simplified as $\mu \left(  \mathrm{H}^1_{\mathrm{Iw}}(\mathbb{Q}_S/\mathbb{Q}_\infty, T_{f,i}(k-r) ) / \Lambda_i \mathbf{z} \right)  = 0$ when $\mathbf{z} = \mathbf{z}_{\mathrm{Kato}}(f, i, k- r)$.
\end{rem}
\begin{proof}[Proof of Theorem \ref{thm:main-theorem}.(2)]
Now we pin down $\Sigma_0 = S' = \lbrace \ell : \ell \mid N_f \cdot N_g \rbrace$.
If we choose
$$\mathbf{z} = \mathbf{z}^{\Sigma_0}_{\mathrm{Kato}}(f, i, k- r)  \textrm{ or } \mathbf{z}^{\Sigma_0}_{\mathrm{Kato}}(g, i, k- r) ,$$
then the LHS of (\ref{eqn:invariant-under-congruences}) depends only on $\overline{\rho}$ and $\Sigma_0$ due to Theorem \ref{thm:congruences-kato-elements} under the $\mu = 0$ assumption (Corollary \ref{cor:mu-invariants-vanish-in-families} and Proposition \ref{prop:comparison-iwasawa-invariants}.(1)).
Thus,  the following equalities are immediate.
\begin{align*}
  & \lambda \left( \dfrac{ \mathbb{H}^1(j_* T_{f,i}(k-r)) }{ \mathbf{z}_{\mathrm{Kato}}(f, i, k- r) } \right)
- \lambda \left( \mathbb{H}^2(j_* T_{f,i}(k-r)) \right) \\
= & \lambda \left( \dfrac{ \mathbb{H}^1(j_* T_{f,i}(k-r)) }{ \mathbf{z}^{\Sigma_0}_{\mathrm{Kato}}(f, i, k- r) } \right)
- \lambda \left( \mathbb{H}^{2, \Sigma_0}( j_* T_{f, i}(k-r) )  ) \right) & \textrm{(\ref{eqn:iwasawa-invariants-kato-elements}) and Proposition \ref{prop:iwasawa-invariants-H^2-imprimitive}}  \\
= & \lambda \left( \dfrac{ \mathbb{H}^1(j_* T_{g,i}(k-r)) }{ \mathbf{z}^{\Sigma_0}_{\mathrm{Kato}}(g, i, k- r) } \right)
- \lambda \left( \mathbb{H}^{2, \Sigma_0}( j_* T_{g, i}(k-r) ) \right) & \textrm{Proposition \ref{prop:comparison-iwasawa-invariants}.(2) }  \\
= & \lambda \left( \dfrac{ \mathbb{H}^1(j_* T_{g,i}(k-r)) }{ \mathbf{z}_{\mathrm{Kato}}(g, i, k- r) } \right)
- \lambda \left( \mathbb{H}^2(j_* T_{g,i}(k-r)) \right) & \textrm{(\ref{eqn:iwasawa-invariants-kato-elements}) and Proposition \ref{prop:iwasawa-invariants-H^2-imprimitive}} .
\end{align*}
Thus, Theorem \ref{thm:main-theorem}.(2) follows.
The final statement is independent of the choice of $\Sigma_0$.
\end{proof}
\begin{rem} \label{rem:iwasawa-invariants-H^2-2}
In \cite{gv} and its various successors, it is essential to identify the $\lambda$-invariants of Selmer groups over the Iwasawa algebra and their $\mathbb{Z}_p$-coranks (under the $\mu=0$ assumption) to apply the congruence argument in an appropriate setting.
In order to do this, one needs to prove that $\mathrm{Sel}(\mathbb{Q}_{\infty}, A_{f,0}(1))^\vee$ (or its variant) has no non-zero finite $\Lambda_0$-submodule. See \cite[Propositions 2.5 and 2.8]{gv} for detail.
In our setting, we do not expect such a statement (for $\mathbb{H}^2(j_*T_{f, i}(k-r))$) holds in general as described in Remark \ref{rem:iwasawa-invariants-H^2}. A similar flavor can also be found in \cite{hachimori}.
\end{rem}

\section*{Acknowledgement}
This project is inspired from Project A\footnote{\href{http://swc.math.arizona.edu/aws/2018/2018SkinnerProjects.pdf}{http://swc.math.arizona.edu/aws/2018/2018SkinnerProjects.pdf}} of Skinner's project group at Arizona Winter School 2018.
We deeply appreciate Chris Skinner and Francesc Castella for wonderful lectures and great helps during the school.
We also thank to other members of the project group: Takahiro Kitajima, Rei Otsuki, Sheng-Chi Shih, Florian Sprung, and Yiwen Zhou.
We thank David Loeffler and Sarah Zerbes very much for the helpful suggestions and discussions in the Arizona Winter School.
We thank Olivier Fouquet, Antonio Lei, David Loeffler, and Xin Wan for helpful comments and Takenori Kataoka for pointing out a gap in an earlier version of this paper.
We deeply appreciate the referee for various helpful and valuable suggestions and comments.
Following suggestions and comments, the exposition is significantly improved and a number of inaccuracies are corrected.
All the authors contributed equally.

\bibliographystyle{amsalpha}
\bibliography{library-klp}
\end{document}